\documentclass[12pt]{amsart}
 \usepackage{mathrsfs}
\usepackage{amssymb}
\usepackage[T1]{fontenc}
 \usepackage{amsthm, amsxtra, latexsym, mathrsfs, amscd}
 \usepackage{graphicx}
\usepackage{epsfig}
\usepackage{amsfonts}
\usepackage{amscd}
\usepackage{amsfonts}
\usepackage{color}
\usepackage{tikz}
\usepackage{psfrag}
\usepackage{pstricks}
\usepackage{mathdots}
\usepackage{psfrag}
\usepackage{amsthm,amsmath,amssymb,graphics,hyperref}
\usepackage{amsmath,amssymb,hyperref} 
\usepackage{enumerate}
\usepackage[titletoc,title]{appendix}
\input xy
\xyoption{all}
\usepackage{ulem}

\theoremstyle{plain}
\newtheorem{theorem}{Theorem}[section]

\newtheorem{proposition}[theorem]{Proposition}
\newtheorem{lemma}[theorem]{Lemma}
\newtheorem{corollary}[theorem]{Corollary}
\theoremstyle{definition}

\newtheorem{definition}[theorem]{Definition}

\newtheorem{remark}[theorem]{Remark}

\newtheorem{example}[theorem]{Example}

\newcommand{\R}{\mathbb{R}}

\newcommand{\ra}{{\rightarrow}}

\let\cal\mathcal

\def\g{\gamma}

\def\P{\mathbb{P}}

            \let\g=\gamma

\let\G=\Gamma

\begin{document}
\title [ Convex Projective Structure]
      {Thermodynamic Formalism on  SRB Measures of Compact Convex Projective manifolds, Geometry and Counting }
        \author{Patrick Foulon and Inkang Kim}

        \date{}
        \maketitle
        \tableofcontents
\begin{abstract}
We study the Hilbert geodesic flow on a closed convex real projective $n$-dimensional manifold. We use the thermodynamic formalism for this dissipative Anosov flow to obtain a geometric description of the SRB measure. We also deduce new counting properties for the induced representations of the fundamental groups in  $SL(n+1, \R)$.  
\end{abstract}
 \footnotetext[3]{Research by Inkang Kim is partially supported by the grant RS-2026-25468457 and KIAS Individual Grant (MG031408), and he also thanks CIRM and IHP for the warm support  during his
visit.}

\section{Introduction} 

\subsection{Flat real projective structures} Let $M$ be a closed oriented $n$-dimensional manifold equipped with a flat real projective structure determined by an atlas of charts  into $P(\mathbb{R}^{n+1})$ whose changes of coordinates are  in $SL(n+1,\mathbb{R})$ whose projective action is identified with  $PGL(n+1,\mathbb{R})$. As a result, we have a holonomy representation $\rho:\pi_1(M)\to SL(n+1,\mathbb{R})$ of the fundamental group $\pi_1(M)$ of $M$, and a developing map $D:\widetilde{M}\to P(\mathbb{R}^{n+1})$ from the universal cover $\widetilde{M}$ of $M$ to the projective space $P(\mathbb{R}^{n+1})$. 

The image $\Omega=D(\widetilde{M})$ of a developing map is an open subset of $P(\mathbb{R}^{n+1})$ which is divisible by the torsion free discrete subgroup $\Gamma=\rho(\pi_1(M))$ of $SL(n+1,\mathbb{R})$ such that the quotient $\Omega/\Gamma\simeq M$ is compact. Note that the natural action of $\Gamma$ preserve the cone $\mathcal{C}\subset\mathbb{R}^{n+1}$ spanned by $\Omega$ and the boundary $\partial\Omega\subset P(\mathbb{R}^{n+1})$ of $\Omega$. 

We say that a flat real projective structure of $M$ is convex whenever  $\Omega$ is convex. Furthermore, a flat real projective structure of $M$ is properly convex (resp. strictly convex) whenever $\Omega$ is a convex (resp. strictly convex) subset of a compact convex set included in an affine chart of $P(\mathbb{R}^{n+1})$. 

\subsection{Hilbert metric} The Hilbert metric of a properly convex open subset $\Omega\subset P(\mathbb{R}^{n+1})$  can be defined in terms of the cross-ratio as follows. Given distinct points $a, b\in\Omega$, the oriented line $\ell$ from $a$ to $b$ intersects $\partial\Omega$ at two points $a^-, b^+\in\partial\Omega$ labeled in such a way that $a^-$, $a$, $b$ and $b^+$ appear in this order on $\ell$. The Hilbert distance $d_{\Omega}(a,b)$ between $a$ and $b$ is 
$$d_{\Omega}(a,b) = \frac{1}{2}\log\left(\frac{|b^+-a|}{|b^+-b|}\cdot\frac{|b-a^-|}{|a-a^-|}\right),$$ 
where $|.|$ is a Euclidean norm in an affine chart containing $\Omega$.  

The intersections of $\Omega$ with projective lines are geodesics of the Hilbert metric $d_{\Omega}$. Moreover, these straight line segments account for all geodesics of $d_{\Omega}$ whenever $\Omega$ is strictly convex: in particular, there is a unique geodesic of the Hilbert metric passing through a given pair of distinct points of $\Omega$ when $\Omega$ is strictly convex. 

\begin{remark} If $\Omega\subset P(\mathbb{R}^{n+1})$ is  an ellipsoid, then its Hilbert metric coincides with the usual hyperbolic metric.   
\end{remark} 

The Hilbert metric $d_{\Omega}$ is invariant under the projective action of the elements $g\in SL(n+1,\mathbb{R})$ preserving the cone $\mathcal{C}$ spanned by $\Omega$. In particular, a closed, strictly convex, flat, real projective manifold $M$ has a Hilbert metric. 

In general, the Hilbert metric $d_{\Omega}$ is  Finslerian but not Riemannian. More concretely, Benoist \cite{B1} showed that a properly convex open $\Omega\subset P(\mathbb{R}^{n+1})$ which is divisible by a torsion free discrete subgroup $\Gamma\subset SL(n+1,\mathbb{R})$ is strictly convex if and only if $\partial\Omega$ is $C^{1+\kappa}$ for some $\kappa > 0$. In particular, the Hilbert metric $d_{\Omega}$ of a strictly convex, open, divisible subset $\Omega\subset P(\mathbb{R}^{n+1})$ comes from the Finsler structure associated with the continuous family of norms 
$$\|v\|_x = \frac{1}{2} \left(\frac{1}{|x^+-x|} + \frac{1}{|x-x^-|}\right) |v|$$ 
on the tangent spaces $T_x\Omega$ to $x\in\Omega$, where $|.|$ is the Euclidean norm and $x^-$ and $x^+$ are the intersection points of $\partial \Omega$ with the line passing through $x$ in the direction $v$. Again, this Finsler norm is $\Gamma$ invariant and induces a Finsler norm on the quotient manifold. If $\Omega$ is not an ellipsoid, then $\partial \Omega$ is not $C^2$  and the Finsler structure above is not Riemannian. 

If $\Omega$ is a strictly convex open subset of $P(\mathbb{R}^{n+1})$, then the geodesics of its Hilbert metric $d_{\Omega}$ define a $C^{\infty}$ one-dimensional foliation. Since this foliation is invariant under the action of any element of $SL(n+1,\mathbb{R})$ preserving $\Omega$, we see that a strictly convex, flat, real projective structure on a closed manifold $M$ leads to a natural 
one-dimensional foliation of $HM= (T\Omega\setminus\{0\})/\mathbb{R}_+^*$ that we may parametrize by the Hilbert distance to get the geodesic flow which is time-reversible.


\subsection{Anosov flow}
It is proved by Benoist \cite{B1},  that this foliation consists of the orbits of a topologically mixing Anosov flow $(\varphi_t)_{t\in\mathbb{R}}$, namely, the geodesic flow of the Hilbert metric of $M$. In the same paper,  Benoist proves  that  the boundary $\partial\Omega$, is, $C^{1+\kappa}$ for some $\kappa>0$ and this (Anosov) geodesic flow has the same regularity.  There exists a Lebesgue absolutely continuous $\varphi_t$-invariant measure if and only if the Hilbert metric is Riemannian. 
 The image of the representation  $ \rho ( \pi_1(M)) $ is Zariski dense in $PSL(n +1; \R)$, except for the Riemannian case. \\

\subsection{Lyapunov exponents and geometry of the boundary}\label{ss.Crampon} Let $M$ be a closed manifold with a flat real projective structure induced by a strictly convex open subset $\Omega\subset P(\mathbb{R}^{n+1})$ which is divisible by $\Gamma\subset SL(n+1,\mathbb{R})$. The geodesic flow $\varphi_t$ with respect to Hilbert metric on the homogenous bundle $HM = (TM\setminus\{0\})/\mathbb{R}_+^*$ is Anosov, i.e., its differential preserves a decomposition 
$$THM = E^s\oplus \mathbb{R}\cdot X\oplus E^u$$ 
where the vectors in $E^s$ are uniformly contracted in the future, $X$ is the generator of the geodesic flow, and the vectors in $E^u$ are uniformly contracted in the past. In general, this decomposition can be further refined at Oseledets regular points $(x,v)\in HM$: by definition, this means that we can find Lyapunov exponents $\chi_p^u(x,v)> \dots >\chi_1^u(x,v)>0 > \chi_1^s(x,v) > \dots > \chi_q^s(x,v)$ and two decompositions $E^s = E^s_1(x,v)\oplus \dots\oplus E^s_q(x,v)$ and $E^u = E^u_1(x,v)\oplus \dots\oplus E^u_p(x,v)$ in Oseledets subspaces such that 
$$\lim\limits_{t\to\pm\infty}\frac{1}{t}\log\|D\varphi_t(x,v) Z_k^u\|_{\varphi_t(x,v)} = \chi_k^u(x,v) > 0$$ 
and 
$$\lim\limits_{t\to\pm\infty}\frac{1}{t}\log\|D\varphi_t(x,v) Z_l^s\|_{\varphi_t(x,v)} = \chi_l^s(x,v) < 0 $$ 
for all $Z_k^u\in E_k^u(x,v)\setminus\{0\}$, $Z_l^s\in E_l^s(x,v)\setminus\{0\}$, $1\leq k\leq p$, $1\leq l\leq q$. By Oseledets theorem, the set of Oseledets regular points has full measure with respect to any $(\varphi_t)_{t\in\mathbb{R}}$-invariant probability measure. 

Interestingly enough, Crampon \cite{cra1} proved that the positive Lyapunov exponents of Oseledets regular points are related to the infinitesimal geometry of $\partial\Omega$. More precisely, let $0< \chi_1^+(x,v)<\dots<\chi_p^+(x,v)$ be  the positive Lyapunov exponents of the Oseledets regular point $(x,v)\in THM$ and $x^+\in\partial\Omega$ be the intersection between $\partial\Omega$ and the oriented line determined any lift of $(x,v)$ to $T\Omega$. From  \cite{cra1} there exists a decomposition $T_{x^+}\partial\Omega=F_1\oplus\dots\oplus F_p$  such that the germ of function $f:(T_{x^+}\Omega,0)\to (\mathbb{R},0)$ whose graph $\{u+f(u)n(x^+)\}$ (where $n(x^+)$ is a normal vector to $\partial\Omega$ at $x^+$) describes $\partial\Omega$ nearby $x^+$ has the following property:  
$$\alpha((x,v),w_i):=\lim\limits_{h\to 0}\frac{\log((f(h w_i)+f(-h w_i))/2)}{\log|h|} = 2/\chi_i^+(x,v)$$ 
for all $w_i\in F_i\setminus\{0\}$, $1\leq i\leq p$. In other terms, the infinitesimal bending of $\partial\Omega$ at $x^+$ in different directions is dictated by the list of positive Lyapunov exponents along any geodesic ray converging to $x^+$ in the future.

\subsection{Lyapunov exponents along periodic orbits}
Periodic orbits of the Hilbert flow were studied by Benoist \cite{B1} and their Lyapunov exponents were computed by Crampon \cite{crampon}. More concretely, consider a closed orbit $\gamma_g$ of the Hilbert geodesic flow associated to a group element $g\in \Gamma$ and denote by $\ell(g)$ the length of $\gamma_g$ (with respect to the Hilbert metric). After Benoist, the element $g$ is bi-proximal.
 
\begin{proposition} \label{pro:exp-periodic-orbit} Let  $\g\in \Gamma$. The eigenvalues of the matrix associated to $ \g$ have moduli (possibly repeated with multiplicity) of the form 
  $ \lambda_1 > \dots\geq \lambda_i \geq \dots > \lambda_{n+1}$. Moreover, if $E_i$ denotes the generalised eigenspace associated to eigenvalues with the same modulus of $\lambda_i$, then
the positive Lyapunov exponents  $\chi^+_i(\gamma)$, $ 1 \leq i\leq n-1$ along the periodic orbit corresponding to $\gamma$  are given by
$$\chi^+_i(\gamma)=2\frac{\log \lambda_1-\log\lambda_{i+1}}{\log \lambda_1-\log\lambda_{n+1}}=\frac{\log(\lambda_1/\lambda_{i+1})}{\ell(\gamma)}$$ where $\ell(\gamma)=\frac{1}{2}(\log\lambda_1-\log\lambda_{n+1})$ is the length of the closed geodesic corresponding to $\gamma$.
These formulas imply in particular that the Lyapunov exponents are such that $ 0 < \chi^+_1(\gamma) \leq \chi^+_2(\gamma) \dots <2$.
\end{proposition} 
As shown in  \cite{B1}, the boundary $\partial \Omega$ is also $\beta $-convex and Benoist defined the  $\beta$-convexity coefficient of an element $\gamma \in \G$ 
$$ \beta(\gamma)=\frac{\log \lambda_1-\log\lambda_{n+1}}{\log \lambda_1-\log\lambda_{2}} = \frac{2}{\chi_1^+(\g)}.$$

Guichard \cite{Gu} showed that for a strictly convex divisible domain $\Omega$ the $\beta$ convexity of the boundary is controlled by the $\beta$-convexity of the periodic orbits via 
$$\beta_\Omega=\sup_\gamma \beta(\gamma) .$$ 
Hence the  $\beta$ convexity of the boundary is linked to the infimum of the smallest positive Lyapunov exponents;

\begin{equation}
\label{betai}
\beta_{\Omega}=\frac{2}{\inf_{\g \in \G}\chi_1^+(\gamma)}.
\end{equation}

 Crampon showed that the closure of the set of positive Lyapunov exponents for all the periodic orbits is a closed interval $[1-a, 1+a]\subset (0,2)$.

All these quantities are invariant under affine and projective transformations. In dimension 2,  approximately $\alpha$-regular for $\alpha$ finite means that the function behaves like $|t|^\alpha$ near the origin, i.e., for small $\epsilon>0$ and small $|t|$, $|t|^{\alpha+\epsilon}\leq  \frac{f(t)+f(-t)}{2}\leq |t|^{\alpha-\epsilon}$, see
Lemma 3.4.3 in \cite{Cr}.

 \subsection{SRB measure} \label{SRB}
It is known that the Bowen-Margulis measure maximizes the entropy among the geodesic flow invariant probability measures \cite{crampon}. This entropy is equal to the topological entropy $h_{top}(\phi)$  of the geodesic flow $\phi$ \cite{crampon} and it is also equal to the exponential growth of the lengths of  closed geodesics:
$$h_{top}(\phi) =\lim_{R\ra\infty} \frac{\log \#\{[\gamma]\in [\Gamma]:\ell(\gamma)\leq R\}}{R},$$ where $[\gamma]$ represents the free homotopy class of loops represented by $\gamma$. For the strictly convex real projective case, this topological entropy is also equal to the volume entropy of the Hilbert metric \cite{crampon}.  It was proved by Crampon \cite{crampon} that in the compact case the entropy of Bowen-Margulis measure of strictly convex real projective manifolds of dimension $n$ is smaller than $n-1$ with equality only in the hyperbolic case. Regarding volume entropy, this result was then extended to properly convex set  by Tholozan  (Theorem 2 in \cite{Tho}).

The Sinai Ruelle Bowen  measure is another very important  invariant probability measure which is so-called physical and hence carries the main information about statistics along the orbits. In smooth Finsler or Riemannian examples that measure coincides with  the Liouville normalized measure. A topologically transitive Anosov flow admits a unique SRB measure, see \cite{Young} for a good survey. The following are  some characterizations of the SRB measure for the flow $\phi$ on $HM$.
\begin{enumerate}
\item  The measure attains the equality in Ruelle inequality $$h_\mu(\phi)=  \int \chi^+ d\mu,$$ where $\chi^+=\sum \text{dim} E_i\cdot \chi_i^+$ denotes the measurable sum of positive Lyapunov exponents.
\item  There exists a set $V$ of full Lebesgue measure
such that for each continuous function $f:HM\ra \R$ and for every $x\in V$,
$$\lim_{T\ra\infty} \frac{1}{T}\int_0^T f(\phi^s(x))ds=\int f d\mu_{SRB}.$$
\item   The measure has Lebesgue absolutely continuous conditional measures on unstable manifolds. This characterization is due to 
Ledrappier-Young \cite{LY}.
\end{enumerate}

For closed Riemannian locally symmetric spaces, Bowen-Margulis measure coincides with SRB measure, but in general they are quite different. From 3 above, if there exists an invariant normalized volume form $\omega$, usually called Liouville volume form in the case of unit bundles, then its associated probability measure coincides with the SRB measure.  Even  when there is no invariant  volume, SRB measures are somehow  invariant measures most compatible with volume. They provide a mechanism for explaining how local instability on attractors can produce coherent statistics for orbits starting from large sets in the basin. See the survey article \cite{Young}. In addition, we proved in a previous paper \cite{PK}
that the  map $h_{SRB}:{\cal P}_n(M)\rightarrow \mathbb R$ is continuous on ${\cal P}_n(M)$,  the space of strictly convex real projective structures on $M$.\\

\section{Results} 

%
To understand the geometric properties of the SRB measure we need to consider the dual convex set $$\Omega^*=\{f\in V^*: \forall   v \in \overline{\Omega}\subset V, f(v)>0\},$$ with its boundary $\partial \Omega^*$ which can be identified with the set of tangent hyperplanes to the boundary, and the normal map
$$n: \partial \Omega \rightarrow \partial \Omega^* $$
mapping a point of the boundary to the tangent space to the boundary at that point. In the sequel, $ \sigma$ denotes the antipodal map  $ \sigma (x,[v]) =( x,[-v] )$ defined on $HM$  and $\mu_{\partial \Omega}$ denotes a Lebesgue measure on the boundary. We will denote by $ c_{\mu_{\partial\Omega}}(\g,\g^+)$ the period function of that measure (see section \ref{period} for definition).
The following theorem unveils the geometry of the transverse measure to the unstable foliation that arises in the SRB measure.


\begin{theorem}\label{main} Let $M$ be a closed manifold equipped with a strictly convex projective structure $ M= \Omega/ \rho(\pi_1(M)) $. 
\begin{enumerate}
\item 
The SRB measure on $HM$ can be written as a local product according to the unstable and stable foliations as 
$$ d\mu_{SRB} = dt \otimes d \nu^{u,L}  \otimes d\nu^s $$  
The conditional measures on the unstable leaves $ dt \otimes\nu^{u,L}$ are absolutely continuous to a Lebesgue measure with non vanishing H\" older densities.
\item The transverse measure to the unstable foliation $d\nu^s$,  is in the measure class  
of the measure  $ \sigma_* n_*^{-1}(\mu_{\partial\Omega^*})$, the pullback by the normal map of a Lebesgue measure  on the dual convex boundary. This measure is in the Lebesgue class if and only if the convex is an ellipsoid.

\item The $\Gamma$-invariant SRB geodesic current on $\partial^{(2)}\Omega= \partial\Omega \times \partial\Omega -\Delta$ for the geodesic $( p,q) $ is of the form
$$G(p,q) \times d n_*^{-1}(\mu_{\partial\Omega^*})(p) \otimes d\mu_{\partial\Omega}(q) $$
with $G$ a non-vanishing  H\"older function on $\partial^{(2)}\Omega$.
\item $G(p,q)$ is a GPS system.

\end{enumerate}
\end{theorem}
\begin{remark}
From the last item in the theorem, we observe that any property  almost sure with respect to SRB measure that can be read on the boundary at infinity is true on a set of measure one with respect to a Lebesgue measure on $\partial\Omega$. 
The geometric description of the SRB measure that we provide in this paper could  lead  to another proof of its continuity  with respect to the representation without using the deep result of Contreras \cite{Con}.\\
\end{remark}
Following our geometric construction, we also obtain  some results for the periods of the Lesbegue measure on the boundary $\partial \Omega$.
We have the following proposition whose proof will be given  in Sections \ref{beta} and \ref{digression}.

\begin{proposition}
\label{minimum}
\begin{enumerate}
\item The periods of $\mu_{\partial\Omega} $, are  given for any $\g \in\pi _1(M)$ by 
    $$ c_{\mu_{\partial\Omega}}(\gamma,\gamma^+) = n\log(\lambda_1(\rho(\g)) +\log( \lambda_{n+1}(\rho(\g))$$
\item  There exists  a positive constant $K(\rho) >0$ depending only on the representation $\rho$ such that 
$$ \inf_{\g} \frac{c_{\mu_{\partial\Omega}}(\gamma,\gamma^+)}{  \ell(\gamma)} = K(\rho)  $$ where $
\ell(\gamma)$ is the Hilbert length of $\rho(\gamma)$.
\item  The entropy $h_{SRB}$ of the SRB-measure is greater or equal to $K(\rho)$.
\end{enumerate}
\end{proposition}

%
\begin{remark}If the projective structure is Riemannian, any hyperbolic isometry has
eigenvalues $\lambda> 1 > \cdots > \lambda^{-1}>0$. Hence
$$c_{\mu_{\partial\Omega}}(\g, \g^+)=n \log\lambda- \log\lambda=(n-1)\log\lambda.$$ Furthermore, $\ell(\g)=\log\lambda$, hence $K(\rho)=n-1$. Consequently
$$h_{SRB}\geq n-1$$ by the third item of the previous proposition. Since the topological entropy is $n-1$ for a closed hyperbolic $n$-manifold, $h_{SRB}=n-1$, and SRB measure agrees with Bowen-Margulis measure.
\end{remark}
To get some counting estimates for the SRB measure we extend the result of Ledrappier \cite{Le} (Corollary 1) and give a full proof for any 
Gibbs measure  with vanishing pressure as long as its periods are positive. Using then the synchronization of the flow as introduced by B. Parry \cite{Pa} and then extended by A. Sambarino \cite{Sa}, together with results of Pollicott for suspension flow over mixing subshifts of finite time, we prove
\begin{theorem}
\label{growth}
Let $\mu_{\partial\Omega} $ be a Lebesgue measure on the boundary $\partial\Omega$ of a divisible convex set $\Omega $.
Then the growth of the periods of $\mu_{\partial\Omega} $, is such that
\begin{align}
\begin{split}
1 = & \lim_{r \rightarrow \infty}sup {1\over r} \log \# \lbrace  [\gamma]  ;  r-1\leq  c_{\mu_{\partial\Omega}}(\gamma,\gamma^+) < r \rbrace \\
  & =
  \lim_{r \rightarrow \infty}{1\over r} \log \# \lbrace  [\gamma]  ;   c_{\mu_{\partial\Omega}}(\gamma,\gamma^+)  \leq r \rbrace  .
  \end{split}
\end{align}
\noindent Indeed, the growth of  $ N(t) = \# \lbrace  [\gamma]  ;   c_{\mu_{\partial\Omega}}(\gamma,\gamma^+)  \leq t \rbrace $ is controlled by 
\begin{equation}
 \lim_{t \rightarrow \infty} te^{-t} N(t) = 1. 
\end{equation}
\begin{remark}
 In section 5, we observe that the first property holds true for any Gibbs measure with vanishing pressure as long as its periods are positive. 
 \end{remark}
\end{theorem}

{\bf Plan of the paper}: We recall some basic yet important facts about Anosov flow, Hopf parametrization, quasi-invariant measures, cocycles in sections \ref{pre} and  \ref{thermo}. For a given quasi-invariant measure $\mu$ on $\partial\Omega$, we construct an equivariant measure $\mu_x, x\in\Omega$ with the same period as $c_\mu$. In section \ref{pot}, it is explained how to construct a potential $F_\mu$ with the period
$c_\mu(\gamma,\gamma^+)=\int_\gamma F_\mu$.
In section \ref{srb}, it is explained how to obtain a SRB measure using $\bar\mu=n^{-1}_*\mu_{\partial \Omega^*}$. It satisfies an important property that
$c_\mu(\gamma,\gamma^+)=c_{\bar\mu}(\gamma^{-1},\gamma^-)$, and $F_{\bar\mu}$ and $F\circ\sigma$ are cohomologous where $\sigma$ is an antipodal map
on $HM$. In section \ref{current}, using Ledrappier-Pollicott-Schapira cocycle, $B^{F_\mu}$, and the H\"older cocyle $c_{F_\mu}$, we show that there exists a Gromov product $(\xi,\eta)_o$  on $\partial^{(2)}\Omega=\{(x^-,x^+)\in\partial\Omega\times\partial\Omega: x^-\neq x^+\}$ for a fixed base point $o\in \Omega$ such that 
$$e^{-(\xi,\eta)_o} d\mu(\xi)\otimes d\bar\mu(\eta)$$ is a $\Gamma$-invariant geodesic current where $(c_{\bar\mu},c_\mu, -(\xi,\eta)_o)$ is a continuous Gromov-Patterson-Sullivan (abbreviated by GPS) system. Finally in section \ref{count}, we show that for
$ N(t) = \# \lbrace  [\gamma]  ;   c_\mu(\gamma,\gamma^+)  \leq t \rbrace $,
\begin{equation*}
 \lim_{t \rightarrow \infty} te^{-t} N(t) = 1,
\end{equation*} using the reparametrization technique developed by Sambarino, and
the prime number theorem by Parry-Pollicott.
%
%
%

\section{preliminaries}\label{pre}

Let $M$ be a closed manifold with  projective structure whose image of the developping map is  a strictly convex open subset $\Omega\subset P(\mathbb{R}^{n+1})$ which is divisible by the discrete subgroup $\rho(\pi_1(M)) =\Gamma\subset SL(n+1,\mathbb{R})$. The geodesic flow $\varphi_t$ with respect to Hilbert metric on the homogenous bundle $HM = (TM\setminus\{0\})/\mathbb{R}_+^*$ is Anosov, i.e.,
\begin{enumerate}
 \item Its differential $d\varphi_t$ preserves a splitting of the tangent bundle of $HM$
\begin{equation}
THM = E^s\oplus \mathbb{R}\cdot X\oplus E^u 
\end{equation}
where the vectors in $E^s$ are uniformly exponentially contracted in the future, $X$ is the generator of the geodesic flow, and the vectors in $E^u$ are uniformly exponentially contracted in the past for some Riemannian metric on $HM$. 

\item The stable distribution $ E^s\oplus \mathbb{R}\cdot X$ has the same regularity as the boundary that is  $C^{1+\kappa}$ and integrates into the so-called stable foliation $\mathcal {F}^s$. The leaf $ W^s_z$ of a point $z\in HM$ is a $C^{\infty}$ dense immersed manifold. But the foliation is transversely $C^{\kappa}$. Similar properties are true for  the unstable foliation $\mathcal {F}^u$ associated with the unstable distribution  $\mathbb{R}\cdot X\oplus E^u$.
\item The strong stable distribution $ E^s$ has the same regularity as the tangent space of the boundary that is  $C^{\kappa}$ and integrates into the so-called strong stable foliation $\mathcal {F}^{ss}$. The leaf $ W^{ss}_z $ of a point $z\in HM$ is a $C^{1+\kappa}$ immersed manifold. But the foliation is transversely $C^{\kappa}$. Similar properties are true for the strong unstable foliation $\mathcal {F}^{su}$ together with the strong unstable distribution  $E^u$.
\item The strong stable leaf of a point $z = (x,[v]) \in HM$ is also the set $W_z^{ss}= \lbrace z' \in HM \vert  \lim_{t\rightarrow +\infty} d(\varphi_t (z), \varphi_t (z'))=0 \rbrace $ with $d$, a distance on $HM$ induced from the Hilbert Finsler norm.

\end{enumerate}
We will have to consider all the lifts of these objects on $H\Omega = (T\Omega\setminus\{0\})/\mathbb{R}_+^*$
 where the geometric description is much simpler. We will use the same notations when this does not create a confusion and will take advantage of the equivariance with respect to the action of the group $\Gamma$. 
 As seen before, any point $ z = (x ,[v]) \in H\Omega $ fixes a unique geodesic  with image $c^{x^+}_x = c_{x,v}$ with $v$ the vector representing $[v]$
 such that $\|v\|_x =1$. Given a point $x \in \Omega$ we introduce the family of maps 
 $\tau_x : H_x\Omega \rightarrow  \partial \Omega$ 
 by $ \tau_x([v])= x^+$ 
 where $x^+$ is the end in the future of the associated geodesic. Map $ \tau_x$ is as regular as the boundary that is $ C^{1+\kappa}$ in the compact strictly convex case. Then  map $ \tau : H\Omega \rightarrow \Omega \times \partial \Omega$ is a fibered homeomorphism above $\Omega$.
 The geometric support of a geodesic is uniquely determined by its end points in the future and in the past $(x^+, x^-)$ that is a point in $  \partial^{(2)}\Omega \equiv \{(x^-,x^+)\in\partial\Omega\times\partial\Omega: x^-\neq x^+\} $. 
 \subsection{Buseman functions}\label{buse}
 There is a classical definition of the Busemann functions that extends without difficulties to the setting of the Finsler Hilbert distance on $H\Omega$.
   From now on,  mark $\Omega$ with some fixed point $o \in \Omega$. Here the Busemann function $B_\xi(o, \cdot):\Omega \ra \mathbb R$ is defined as
$$B_\xi(o, x)=\lim_{t\ra\infty} (t-d(\gamma_o(t), x)),$$ where $\gamma_o$ is a unit speed geodesic from $o$ such that $\gamma_o(\infty)=\xi$.
Here $z$ is an any point in $\Omega$, and the Busemann function is normalized so that $B_\xi(o,o)=0$.\\ 
There is another way to present it namely using  the  one forms obtained by restricting the flow invariant Hilbert form $A^H$ (see \cite{ Cr}) to each stable  leaves. Those forms are only $C^{\kappa}$ but they are  necessarily  closed due to the Anosov contraction  property. Then  up to a constant there exists on each stable  leaf $ W^s_z$ a $C^{1+\kappa}$ function $\tilde B^s_z  : W_z^s \rightarrow \R$ such that $d\tilde B^s_z = - A^H_z$. By definition their  level sets are the strong stable leaves on $H\Omega$.\\  The advantage of this second presentation is that we can generalize it in presence of a potential, see \cite{scha}. It is shown by Benoist that the strong stable distribution is everywhere transverse to the vertical distribution hence any strong stable leaf $W_z^{ss}$ projects over $\Omega $ to  a smooth manifold $S^s (x(z), x^+(z))$ called the horosphere at $x$ with point at infinity $x^+$. The base projection reduced to a strong stable leaf on $H\Omega$ is a smooth diffeomorphism on its image.\\
   From now on,  mark $\Omega$ with some fixed point $o \in \Omega$. 
 The family of stable  Buseman  are the projection of the function $ \tilde B^s$. They depend  $ C^{1+\kappa}$ smoothly on $x$ and are only $C^{\kappa}$ H\"older-continuous  with respect to the  boundary point.  \\
  The Hopf parametrization   $\mathcal{H} : H\Omega \rightarrow   \partial^{(2)}\Omega \times {\mathbb R}$ defined by 
 \begin{equation}
\mathcal{H} (x, [v]) = ( x^-(x, [v]) , x^+(x, [v]) ,  B^s(x,x^-(x, [v])))
 \end{equation}
   
is a  $C^\kappa$ bi-Hölder continuous homeomorphism, and, as noticed before, the two first components are indeed $C^{1+\kappa}$.
In these coordinates it is easy to express the action of the lift of the Hilbert geodesic flow
 \begin{equation}
 \varphi_t( x^- , x^+,  \tau ) = ( x^- , x^+ ,  \tau+t)) 
\end{equation} 

 And of the action of the fundamental group 
 \begin{equation}
 \g ( x^- , x^+,  \tau ) = ( \g x^-, \g x^+, \tau - B^s( \g^{-1}o, x^-)
 \end{equation} 
 \\
 We will follow the notations of Ledrappier (\cite{Le}) and will say that a function $f $ on $ \Omega \times \partial \Omega$ is of class $C_s^k$ if $ f( ., \xi) $ is $C^k$  and for $j\leq k$ the j-jets of the function $ f(. , \xi)$ are Hölder continuous in $\xi$ . Similarly a function $ g $ on $H\Omega$ will be said in $ C^k_s$ if 
 $ f = g \circ \tau^{-1}$ 
is $C_s^k$. It is possible to have the same definitions with the strong unstable foliation. 
\section{thermodynamic formalism}\label{thermo}
\subsection{Hopf parametrization of equilibrium states}\label{ss.SRB-BM} 
 
Recall that an equilibrium measure $\mu$ for a flow $(\varphi_t)_{t\in\mathbb{R}}:X\to X$ with respect to a potential $f:X\to \mathbb{R}$ is a flow-invariant probability achieving the supremum of the quantity $h_\nu(\varphi_1)+\int f d\nu$ among all flow-invariant probability measures $\nu$. In the case of the geodesic flow $(\varphi_t)_{t\in\mathbb{R}}$ of the Hilbert metric on a flat real projective manifold $M$ attached to a strictly convex open subset $\Omega\subset P(\mathbb{R}^{n+1})$ which is divisible by $\Gamma\subset SL(n+1,\mathbb{R})$. \\ We know from Haydn \cite{Haydn} that any equilibrium state $\mu_f$ of the (topologically mixing, Anosov) flow $(\varphi_t)_{t\in\mathbb{R}}$ with
respect to a H\"older potential $f:HM\to\mathbb{R}$ corresponds to a
Radon measure with a \emph{product} structure 
$$d\mu_f= dt\otimes d\nu^s\otimes d\nu^u$$ on $HM$
\\
  The homogeneous bundle of the universal cover$H\tilde M$  is   $C^\kappa$ bi-H\"older continuous homeomorphic to ,
 $\mathbb{R}\otimes
\partial^{(2)}\Omega$, where $\partial^{(2)}\Omega=\{(x^-,x^+)\in\partial\Omega\times\partial\Omega: x^-\neq x^+\}$ is the space of oriented geodesics of the Hilbert metric of $\Omega$. The projections on the two last components are indeed $C^{1+\kappa}$. Furthermore, $\mu_f$ corresponds to a $\Gamma$-invariant Radon measure on $\partial^{(2)}\Omega$ of the form $G(u,v)d\nu^s(u)\otimes d\nu^u(v)$ for a H\"older continuous function $G$ on $\partial^{(2)}\Omega$. 

Among the most well-known equilibrium states, one finds the Bowen--Margulis measure $\mu_{BM}$ associated to the trivial potential $f\equiv 0$ and the Sinai--Ruelle--Bowen (SRB) measure $\mu_{SRB}$ associated to the potential $f=\frac{d}{dt}(\log\det D\varphi_t|_{E^u})|_{t=0}$. For later reference, let us recall that the product measure $G_{SRB} \cdot \nu^s_{SRB}\otimes\nu^u_{SRB}$ on $\partial^{(2)}\Omega$ associated to $\mu_{SRB}$ has the property that $\nu^u_{SRB}$ is absolutely continuous with respect to a Lebesgue measure on $\partial\Omega$.  

\subsection{Quasi-invariant measures, cocycles and periods}
Periodic orbits are, as in the Riemannian case, closely related to the representation 
$\Gamma=\rho(\pi_1(M))$ into the group $SL(n+1,\mathbb{R})$ \cite{B1}. For any element $g$ of this group let us denote
$ \lambda_1(g) \geq \cdots \geq \lambda_{n+1}(g)$ the modulus of the eigenvalues of $g$ repeated with multiplicities  and $l_1(g) \geq \cdots \ \geq l_{n+1}(g)$
the logarithms of these modulus. An element is proximal if $ \lambda_1(g) >  \lambda_2(g)$. Let us denote by $x^+(g)$  the unique attracting fixed point of a proximal element $g$ in $\G \subset SL(n+1,\mathbb{R})$. An element $g$ is said bi-proximal if $g$ and $g^{-1}$ are proximal. From Benoist  we know that for strictly convex divisible convex $\Omega$  every element $g$ of $\Gamma - \lbrace  1\rbrace$ is bi-proximal and fixes a unique geodesic in  $\Omega$, the  straight line  joining $x^-(g) =x^+(g^{-1})$ to $x^+(g)$. Each non-trivial free homotopy class $[g]$ contains a unique closed geodesic. {\bf For simplicity of notation, sometimes we will use the notations
$\gamma^\pm$ for unique attracting and repelling fixed points of $\rho(\gamma)$. We will use $\g^\pm$ and $x^\pm(\g)$ interchangeably, and note that $(\g^{-1})^+=\g^-$}.

\label{period}
In this part we will follow the fruitful approach developed by
Ledrappier \cite{Le}, Hamenst\" adt, Kaimanovich, Benoist, etc where relations between cocycles on the boundary, quasi-invariant measures on the boundary, transverse measures to the stable or unstable foliation - Gibbs measures associated to a potential and Patterson-Sullivan measures, are shown to rely on some cocycles. The paper of Ledrappier deals with geodesic flows of negatively curved Riemannian closed manifolds. But, as we have seen here the Hilbert flow is Anosov,  the underlying geometry of the boundary $\partial \Omega $ and the Hopf parametrization are very similar to the  negatively curved Riemannian case. The regularity for the potentials in \cite{Le} is $ C^2_s$ but in fact by looking at the proofs  we observe that the difference of regularity in our case is due only to the normalization of the generator $X$ of the flow, which forces it to be  $C^{1+\kappa}$ 
and not $C^\infty $ as in the Riemannian case. This lack of regularity  appears only when there is a contraction of a form with $X$ or a Lie derivative along the integral curves of $X$, but no further derivative is needed. So many of  the statements in \cite{Le} may be replaced  with a regularity $C^{1+\kappa}_s$. 
\begin{remark} 
 Nevertheless,  Ledrappier uses some results of harmonic analysis in the third assertion of theorem 3 of \cite{Le}. These properties deeply rely on Riemmannian geometry. In the Hilbert geometry case we find an explicit construction for the SRB measure but we have not  been able to obtain such a nice geometric understanding for  all equilibrium states.
\end{remark}

Let us recall some definitions with some properties.
\begin{proposition}\label{Prop3.1}
A H\"older cocycle  on the boundary is  a map $ c:\pi_1(M) \times  \partial\Omega \rightarrow \R$ , $ c(\gamma,.) \in C^{\alpha}(\partial\Omega,\R)$ such that
\begin{enumerate}
\item $c(\g_2\g_1,\xi) = c(\g_1,\xi) +c(\g_2, \g_1.\xi) $
\item Two cocycles c and $c'$ are cohomologous if there exists a H\" older function  $U$ such that
$$c(\g_,\xi) -c'(\g,\xi) = U(\g.\xi)-U(\xi) $$

\item $ c(\g,x^+(\g))$ depends only on the cohomology class. These numbers are the periods of $c$.
\item Two H\"older cocycles are cohomologous iff they have the same periods  \cite{Le} [Thm 1.a]
\end{enumerate}
\end{proposition}
\begin{remark}
From the definition  we get $ c(1,\xi)= 0$ because from (1) we have $c(1.\g,\xi) = c(\g,\xi)+ c(1,\g\xi) $ . \\
Thus in particular
\begin{equation}
c(\g^{-1}, x^+(\g^{-1})) + c(\g , x^-(\g)) =0 .
\end{equation} 
\end{remark}

\begin{theorem}
 \label{quasi-invariant}
 \cite{Le} [Thm 1.b]
 A  positive finite measure $\mu$ on $\partial\Omega$ is $\G$ quasi-invariant if there exists a H\"older cocycle $c_\mu$ on $\partial \Omega $ such that  for any $\g \in \G$ and any continuous function  $f$,
 \begin{equation}
 \int_{\partial \Omega} f(\g^{-1} \xi) d\mu=  \int_{\partial \Omega} e^{ -c_\mu(\g, \xi)}  f(\xi) d\mu
 \end{equation}  such that
 the pullback measure $ \g ^*(\mu)$ is absolutely continuous
 with respect to $\mu$ with the H\" older
  Radon-Nikodym derivative.
 \begin{enumerate}
  \item
 $   \g ^*(\mu) = e^{-c_\mu(\g,.) } \mu$
 
  \item   The period function of the measure $\mu$ is given by 
  $$ c_\mu(\g,\g^+) $$
for each $\g\in\Gamma$.
   \item It is known that $\mu'$ is absolutely continuous  with respect to $\mu$  if and only if  they have the same periods, i.e.,
   $$c_{\mu'}(\g,\g^+)= c_{\mu}(\g,\g^+) $$ for all $\g\in\Gamma$.
  
\end{enumerate}
  \end{theorem}

\subsection{  Examples of Quasi-invariant Measures  }
Benoist \cite{B1} has studied the periods for some quasi invariant measures. Let us choose  Riemmannian  metrics  on $\partial\Omega$, and on $\partial \Omega^* $  the boundary of the dual convex  $\Omega^*$, with their corresponding  finite volumes $ \mu_{\partial \Omega}$  and  $ \mu_{\partial \Omega^*}$.  The normal map $n: \partial \Omega \rightarrow \partial\Omega^*$  associates  to each point $\xi \in \partial \Omega$  the tangent hyperplane $ n(\xi) =T_{\xi} \partial\Omega$, which is a point on $\partial \Omega^*$. The normal map is $\G$- equivariant that is $ n =^t\g n\g$ .
 \\
\begin{proposition}\label{relation}
[\cite{B1} Lem 6.5, Cor 6.6]  With the notations above the two finite measures $  \mu_{\partial \Omega}$ and  
$ n^{-1}_* \mu_{\partial \Omega^*} $  on $ \partial \Omega$ are $\G$-quasi-invariant. The periods are given by the relations

\begin{equation}
\label{lebesgue}
 c_{\mu_{\partial\Omega} }(\g,\g^+) = n\log(\lambda_1(\g)) +\log( \lambda_{n+1}(\g))  
\end{equation} \\
 The periods are positive, i.e.,
 for any $\g \neq 1$, $c_{\mu_{\partial\Omega} }(\g,\g^+) > 0$.
\begin{equation}
 \label{lebesgue*}
 c_{n^{-1}_* \mu_{\partial \Omega^*}}(\g,\g^+) = -\log(\lambda_1(g)) - n\log( \lambda_{n+1}(\g))  
 \end{equation}
 The periods are positive, i.e.,
  for any $\g \neq 1$, $c_{n^{-1}_* \mu_{\partial \Omega^*}} (\g,\g^+) > 0$.
 It is interesting also to remark that 
 \begin{equation}
 \label{identity}
 c_{n^{-1}_* \mu_{\partial \Omega^*}}(\g,\g^+)  =  c_{\mu_{\partial\Omega} }(\g^{-1},(\g^{-1})^+).
 \end{equation}

\end{proposition}
\begin{proof}
 Because these formulas will play an important role in the sequel and for the sake of the reader we give a short proof.
  For (5) we use the fact that the group $\G$ acts by diffeomorphisms on $\partial\Omega$ insuring that  the  volume measure is $\G$-quasi-invariant. To compute the periods we use  (1) in Theorem
  \ref{quasi-invariant}  which gives the relationship to the Jacobian  of the differential of the action $\g$ on the tangent hyperplane $y^+(\g)$ at the fixed point $x^+(\g)$ 
  $$ \log \det(d\g)_{x^+(\g)} = - c (\g, x^+(\g)) $$
  The calculus of this determinant gives
  $$ \det(d\g)_{x^+(\g)} = \prod \limits_{i=1}^{n-1} \frac{ \lambda_{i+1}(\g)}{\lambda_1(\g)}$$ 
  and using $ 1 = \lambda_1(\g)  \lambda_n(\g)  \prod \limits_{i=1}^{n-1} \lambda_{i+1}(\g) $ we get 
   $$ \det(d\g)_{x^+(\g)}= \frac{1}{\lambda_1^n(\g)  \lambda_{n+1}(\g)}.$$
   All the periods are positive because  for $ 2\leq i\leq n$   the proximality property implies that $ \lambda_1(\g) > \lambda_i(\g)$ hence 
   $ \lambda_1^n(\g) \lambda_{n+1} (\g)>1$.
   
   For  \ref{lebesgue*}  we need to evaluate
    $$ \g ^*n^{-1}_* \mu_{\partial \Omega^*} = \g^{-1} _*n^{-1}_* \mu_{\partial \Omega^*} = (n \g)^{-1}_*\mu_{\partial \Omega^*}.$$
 Note that we have replaced the pullback measure by the pushforward measure of the inverse. 
 The equivariance  implies 
 $$ \g ^*n^{-1}_* \mu_{\partial \Omega^*}= (^t \g^{-1} n)^{-1}_*\mu_{\partial \Omega^*} =  n^{-1}_*. ^t \g _*{\mu_{\partial \Omega^*}}.$$ The last statement is a simple computation. The positivity of all the periods is then immediate.\\
 \end{proof}
 As already observed by Benoist \cite{B1}
  \begin{corollary} For a convex divisible set the two measures above have the same null sets if and only if the convex set $\Omega$ is an ellipsoid.
  \end{corollary}
\begin{proof}Theorem \ref{quasi-invariant} (3) states that the equality of the period is equivalent to having the same null sets. If this is the case  for any $\g \in \G$ we have from the previous proposition 
  $$ \log(\lambda_1(\g)) + \log( \lambda_{n+1}(\g)) = 0  ,$$
  which contradicts the Zariski density of $\G$. In more details, if the convex set is an ellipsoid, then the holonomy representation is conjugate into $SO(n,1)$ in which any hyperbolic isometry has eigenvalues $\lambda, 1,\cdots,1,\lambda^{-1}$. Then it obviously satisfy the above equation.  If the convex set is not an ellipsoid, then its holonomy representation is Zariski dense in $SL(n+1,\mathbb R)$. By Benoist \cite{BGafa}, the limit cone, which is the closure of the cone generated by $(\log(\lambda_1(\g)),\cdots,\log(\lambda_{n++1}(\g)))_{\g\in\Gamma}$, has nonempty interior in $\mathfrak a^+$ where $\mathfrak a^+$ is a fixed positive Weyl chamber of some Iwasawa decompsotion of $SL(n+1,\mathbb R)=KAN$. But if $ \log(\lambda_1(\g)) + \log( \lambda_{n+1}(\g)) = 0 $ is satisfied, it cannot have nonempty interior.
 \end{proof}
 
 \subsection{ Family of equivariant measures on the boundary}
 We again follow \cite{Le}.
 \begin{definition} \label{equivmeasure1} An equivariant family of measures on the boundary $\partial \Omega $ is a map
 $ x\rightarrow \mu_x$ from $\Omega $ to the set of Radon measures on $\partial \Omega $ such that
 \begin{itemize}
 \item for any $\g \in \G$ , $\mu_{\g x}= \g_* \mu_x$ ,
 \item for any $x $,  $y$ in $\Omega$, the measures $\mu_x$ and $\mu_y$ have the same negligible sets and there exists a positive function $k_x$ on $\Omega \times \partial \Omega$
 of class  $C^{1+\kappa}_s$ such that
$$  \frac {d\mu_y}{d\mu_x}(\xi) = k_x (y,\xi ) \ ,\  \mu_x   - a.e. $$
\end{itemize}
\end{definition}
\begin{proposition} [\cite{Le} Proposition 1]
 There exists a one to one correspondance between classes of family of equivariant measures on the boundary and classes of quasi-invariant measures with the same periods.
\end{proposition}
Let us give without proofs the construction in \cite{Le}.
Given $\mu$  a quasi-invariant measure on the boundary, one chooses a point $p \in \Omega$ and a compactly supported $C^{\infty} $-real positive function $\rho$ on $\R$  such that $\rho(0)= 1$, $\rho'(0) = \rho''(0) =0  $, $\rho( t) > 1/2 $ if $ \vert t\vert < 2\ \underset{x \in \Omega}{\sup}\ d_\Omega(x, \G p)$, and we set
\begin{equation}
\label{equivmeasure}
 \mu_x  = \underset{\g \in \G}{\sum} \rho (d_\Omega(x, \g p))\g_*\mu
\end{equation} 
It is proved in \cite{Le} that this family of measures on the boundary is equivariant and its associated cocycle in (\ref{cf}) is cohomologous to the cocycle associated to $ \mu$.
 
\begin{example}
We can use (\ref{lebesgue}) and (\ref{lebesgue*}) to obtain two families of equivariant measures $ \mu_x^L$ and $\mu_x^{L*}$ having the same periods  respectively as in (\ref{lebesgue}) and  (\ref{lebesgue*}).
We see that all the measures in the first equivariant family  is also in the Lebesgue measure class and the second family is  the pullback of one family in the Lebesgue measure class on $\partial \Omega^*$.

\end{example}

Conversely to any family of equivariant measures on the boundary one associates a cohomology class  of H\"older cocycles. This is possible because for any $p \in \Omega$ the measure $\mu_p$ is a quasi-invariant measure on the boundary. The first item of Definition \ref{equivmeasure1} $ \g_* \mu_x =\mu_{\g x}$ and the second item assert that 
$\mu_p$ and $\mu_{\g^{-1}p}$ have the same null sets. Hence the class of the  H\"older cocycle defined below does not depend on the choice of $p$
\begin{equation} 
\label{cf}
 c( \g , \xi) = - \log k_p (\g ^{-1} p, \xi)
\end{equation}
since two measures $\mu_p$ and $\mu_q$ are absolutely continuous for any pair $p$ and $q$, hence the two H\"older cocycles $c_p(\gamma,\xi)$ and $c_q(\gamma,\xi)$ have the same periods by Theorem \ref{quasi-invariant} (3), and in turn by Proposition \ref{Prop3.1} (4), two cocyles are cohomologous. We will denote this cocyle by $c_\mu$.
Thanks to \cite{Le} [ Thm 1.d] two families of equivariant-measures on the boundary have the same null sets if and only if their cocycles have the same periods. \\

\subsection{Associated Potential}\label{pot}
Before giving their associated transverse measures, we need to relate these measures to functions defined on $HM$. There is a classical way to do it, and we may use Equation (\ref{equivmeasure}) and remark that 

\begin{equation}
\label{equivfunction}
d\mu_x = G(x , \xi) d\mu 
\end{equation}
with
\begin{equation}
G_{\mu}(x , \xi)  = \underset{\g \in \G}{\sum} \rho (d_\Omega(x, \g p)) e^{-c_{\mu}(\g^{-1}, \ \xi)}
\end{equation} 
which is a $ C^{1+ \kappa}_s$ function on $\Omega \times \partial \Omega$. Under the action of $\Gamma$, the function is changed according to the rule

\begin{equation}
\label{equivrel}
G(\g x , \g\xi)  = G(x,\xi) e^{-c_{\mu}( \g,\xi)}.
\end{equation} 
We can get a family of exact H\"older one forms on $\Omega $ by defining 
\begin{equation}
 \alpha_{ \xi}= -d^  {\Omega}\log G_{\xi},
\end{equation} 
where the derivative is taken along $\Omega $.
Then it is easy to see that
$$\g^*  \alpha_{ \g \xi}   = \alpha_{ \xi}.$$  
Then we introduce a H\"older function $F_{\mu}$ associated to the measure $\mu $ by

\begin{equation}
\label{ctoF}
 F_{\mu} (x,\xi) = \alpha_{\xi}( X(x,\xi)),
\end{equation}  where $X(x,\xi)$ is the flow vector field at $x$ in the direction of $\xi$.
From above, $F_{\mu}$ is $\Gamma $-invariant and thus descends to the manifold $HM$. 
Using (\ref{equivfunction})  and (\ref{ctoF}) it is easy to observe that for $(x_t, \xi)= \phi_t (x,\xi)$ one has 
\begin{equation}
\label{varmeasure}
\frac {d \mu_{x_t}}{d \mu_x} (\xi) = e^{- \int_0^t F_{\mu}(x_s, \xi) ds}.
\end{equation}
With the usual notations for the periods of the function  and using \ref{equivrel} a simple computation  shows
\begin{equation}
\label{periods}
 \int_{\g} F_{\mu} =  \int_{\g} -i_X d^  {\Omega}\log G_{\xi}= \log \frac {G_{\mu}(x,\g^+)}{G_{\mu}(\g x,\g^+)}= c_{\mu}(\g , \g^+).
\end{equation}

\subsection{Pressure of a function}
Let us recall  the definition of the pressure of a continuous function (\cite{Bo}, \cite{Bo-Ru}) and specialize it to our setting. 
\begin{definition} Let $\ell(\g) = 1/2( \log\lambda_1(\g) -\log\lambda_{n+1}(\g))$ be the Hilbert length of the closed geodesic freely homotopic to $\g$. Following Gurevich,
 the pressure of a continuous function F defined on $HM$ is given by

\begin{equation} 
 \label{P}
P(F) = \limsup_{T \rightarrow  +\infty}  \frac {1 } {T} \log   \underset{[\g] \in [\G] ,T-c <\ell( \g )< T}{\sum} e^{ \int_{\g} F}
=  \limsup_{T \rightarrow  +\infty}  \frac {1 } {T} \log Z(F,T,c)
\end{equation}
where $c$ is a constant sufficiently large.
\end{definition}
\begin{remark}
The pressure of the null function is just the topological entropy in our case. 
Two cohomologous functions have the same pressure which depends only on the periods.
 \end{remark}
\noindent {\bf Examples}:
Using (\ref{periods}) we may express the pressure of the functions $ F_{\mu_{\partial \Omega}}$ and  $F_{n^{-1}_* \mu_{\partial \Omega^*}}$ associated with our examples. \\ Before doing this for these two measures let us remark that  all periods are positive. And from 
(\ref{identity}) we immediately obtain
\begin{equation}
P( - F_{\mu_{\partial \Omega}})=P(- F_{n^{-1}_* \mu_{\partial \Omega^*}})
\end{equation}
But thanks to Ledrappier we can go further.

\subsection{From an equivariant family to a transverse measure} 
\label{transverse}
A family of equivariant measures on the boundary produces  a  transverse measure either to the stable or the unstable foliation.
To this end, let us recall that for any $x\in M $ the fibers $H_xM$ is transverse to both foliations.  The family of measures defined for each $p \in \Omega$  on the fiber $H_p\Omega$ by
\begin{equation}
\nu_p = (\tau_p^{-1})_* \mu_p
\end {equation}
can be pushed down to a family of measures well defined on the fibers $H_xM$. Using 
$ \g\circ \tau_p = \tau_{\g p} \circ \g$ and the equivariance, we have
\begin{equation}
\nu_{\g p} = \g_* \nu_p
\end {equation}
Again using (\ref{equivmeasure}) and  the fact that through point $z\in HM$ of a stable leaf $W^s_z$ there is a transverse fiber, we may use the usual construction to get  a $C^{1+\kappa}_s $ transverse measure to the stable foliation on $HM$.
This gives us a family of measures of class $C^{1+\kappa}$ on the strong unstable leaves. For instance in our first example the transverse measure is in the Lebesgue class with H\"older densities. \\


%
  \section{ SRB measure} \label{srb}
In this section,  we use the two measures (\ref{lebesgue}), (\ref{lebesgue*}) to produce an invariant measure on $HM$. This measure will have the properties of the SRB measure. It would have been tempting to do this  for all Gibbs measures but we cannot a priori apply all the formalism of Ledrappier because the Hilbert flow is not Riemannain and also because of a weaker  regularity.
Our strategy will be to produce a transverse measure to the unstable foliation from (\ref{lebesgue*})  and then show that, at the level of the periods,
that measure compensates those of the transverse measure to the stable foliation deduced from (\ref{lebesgue}). Then we use the Livsic homological theorem to find an invariant measure.\\
\subsubsection{ SRB transverse measure to the stable foliation}
Following Section \ref{transverse} we can now produce  a  family of transverse measures to the stable foliation using $\mu_{\partial\Omega}$. \\
As already observed, all the measures on the fibers are in the Lebesgue class with H\"older densities. This  of course remains true  for the induced measures on the strong unstable leaves. For $z \in HM$ we will denote  by $\nu^{u,L}_z$  the measure induced on the unstable leave $W^u_z$. All these measures are in the same measure class with H\"older densities. Remember that the holonomy of the stable foliation is $C^{1+\kappa}$. 
From (\ref{varmeasure})  
 we have 
\begin{equation}
\label{varmeasureL}
\frac{d{\phi_{-t }}_* \nu^{u,L}_{\phi_t(z)}} {d \nu^{u,L}_z}(z) = e^{- \int_0^t F_{\nu^{u,L}}(\phi_s(z)) ds}
\end{equation} 



\begin{remark}We could also, as usual, have chosen a Riemannian metric on $HM$ and have considered this induced metric on the strong unstable leaves and considered, as usual, the function $f(v) =  \frac{d}{dt} \log det (d\phi_t)_{\vert E^u}$ computed with respect to the corresponding  volume. Because these two family of measures have the same negligible sets the functions $ F_{\mu_{\partial\Omega}}$and $f$ have the same periods  according to Theorem \ref{quasi-invariant} (3) and (\ref{periods}).
\end{remark}

\subsubsection{ SRB transverse measure to the unstable foliation}

Using  now the quasi-invariant measure $n^{-1}_* \mu_{\partial \Omega^*}$, we similarly produce another transverse measure to the stable foliation. We denote by $\nu^{u,L^*}_z$ the measure induced on the unstable leave $ W^u_z$ . 

To get a transverse measure to the unstable foliation we use the $ \sigma $ antipodal map  such that for any $z = (x,[v]) \in HM$ we have
$ \sigma(x,[v])  =(x,[-v]) $ and take advantage of the fact that the symmetry of the Hilbert metric implies that the geodesic flow is reversible.
We may sum up the classical properties we will use;
\begin{enumerate}
\item  $\sigma_* X = - X$
\item $\sigma \circ \phi_t  \circ \sigma = \phi^{-1}_t= \phi_{-t}$
\item $ d \sigma E^s_z =E^u_{\sigma (z)}$
\item For any $z =(x, [v])= (x,\xi)$ , $\sigma (x,\xi)= (x,-\xi)$
\item $ \sigma W^s_z =W^u_{\sigma (z)}$
\item 
If $ (x_t, x^+([v])) = \phi_t (z)$ 
is an orbit of the flow, then $ \sigma (x_t, x^+([v])) = (x_t, x^+([-v])) =(x_t, x^-([v]))$ is an orbit of $-X$ which project down to the same unoriented geodesic. 
\end{enumerate}
Now we define the transverse measure to the unstable foliation as a family of measures on the strong stable leaves according to 
\begin{equation}
 \nu^{s,L^*}_{z} = \sigma_* \nu^{u,L^*}_{{\sigma (z)}}.
\end{equation}
That is for any borel set $B \subset W^s_z$,  $ \nu^{s,L^*}_{z}(B) =   \nu^{u,L^*}_{\sigma (z)}(\sigma^{-1}(B))$.
Now we may remark that using the reversibility of the flow one gets
\begin{align}
\frac{d{\phi_{-t }}_* \nu^{s,L^*}_{\phi_t(z)}} {d \nu^{s,L^*}_z}(z) & =   \frac {d {\phi_{-t }}_*\sigma_* \nu^{u,L^*}_{\sigma \phi_t(z)}}{d \sigma_*\nu^{u,L^*}_{\sigma(z)}} (z) \\
& =  \frac {d {\sigma_*\phi_{t }}_* \nu^{u,L^*}_{\phi_{-t}\sigma (z)}}{d \sigma_*\nu^{u,L^*}_{\sigma(z)}} (z) 
\end{align} 
We can use the fact that if two measures $ \mu_1$ and $\mu_2$ satisfy 
$d\mu_2 = f d\mu_1$ and $ \psi$ is a diffeomorphism then $ d \psi_* \mu_2= f\circ \psi^{-1}   \ \psi_* d\mu_1$
and we apply it with $\sigma$ to the pair of measures 
${\phi_{t }}_* \nu^{u,L^*}_{\phi_{-t}(\sigma (z))}$  and $\nu^{u,L^*}_{\sigma(z)}$.
 Then  from (\ref{varmeasure})
\begin{equation}
\label{evol*}
\begin{split}
\frac{d{\phi_{-t }}_* \nu^{s,L^*}_{\phi_t(z)}} {d \nu^{s,L^*}_z}(z)   & = 
\frac {d {\phi_{t }}_* \nu^{u,L^*}_{\phi_{-t}(\sigma (z))}}{d \nu^{u,L^*}_{\sigma(z)}} (\sigma(z)) \\
& = 
e^{- \int_0^{-t }F_{\nu^{u,L^*}}(\phi_{s}(\sigma(z))) ds}\\
& = e^{ \int_{-t}^0F_{\nu^{u,L^*}}(\phi_{s}(\sigma(z))) ds}
\end{split}
\end{equation} 
Then by (\ref{periods}), and  Proposition \ref{relation} we obtain
\begin{lemma} Let $\g \in \G$ and $z$ a point on the corresponding  periodic orbit of length $T = \ell(\g)  $. Then
\begin{equation}
\frac{d{\phi_{-T }}_* \nu^{u,L}_{\phi_T(z)}} {d \nu^{u,L}_z}(z) = e^{ -c_{ \mu_{\partial \Omega}} (\g, \ x^+(\g))}
\end{equation}
\begin{equation}\label{unst}
\frac{d{\phi_{-T}}_* \nu^{s,L^*}_{\phi_T(z)}} {d \nu^{s,L^*}_z}(z)  = e^{ c_{n^{-1}_* \mu_{\partial \Omega^*}}(\g^{-1} ,\ x^+( \g^{-1}))} 
=  e^{c_{ \mu_{\partial \Omega}} (\g, \ x^+(\g))}
\end{equation}
\end{lemma}

Furthermore we have by comparing with (\ref{unst})
\begin{lemma}
The periods of $F_{\nu^{u,L^*}}$ are the same as those of  $ F_{\nu^{u,L}} \circ \sigma$
\begin{equation}
\int_0^{T }F_{\nu^{u,L}} \circ \sigma (\phi_{s}(z)) ds 
=  c_{ \mu_{\partial \Omega}} (\g^{-1}, \g^-)
\end{equation} 
\end{lemma}
\begin{proof}
$$\int_0^{T }F_{\nu^{u,L}} \circ \sigma (\phi_{s}(z)) ds  =
\int_0^{T } F_{\nu^{u,L}} \circ \phi_{-s} ( \sigma (z)) ds $$
For a $z=(\gamma(0), \gamma'(0))$, $\phi_{-s}(\sigma(z)),   0\leq s\leq T$ travels along $\gamma^{-1}$ toward $-\gamma'(0)$ direction, hence
the integration is performed  on the image by $\sigma$ of the periodic orbit $\g$, i.e $ \g^{-1}$.
\end{proof}
\begin{proposition} 
\label{livsic}
There exists a  H\"older function $ U : HM \rightarrow \R$ differentiable along the flow
such that 
$$F_{\nu^{u,L^*}} =  F_{\nu^{u,L}}\circ \sigma + X U$$

\end{proposition}
\begin{proof}The two functions in the proposition have the same periods and thus by Livsic Homological Theorem they are cohomologous, i.e., there exists a H\"older coboundary $U$ statisfying this equality.\end{proof}
\begin{proposition}  \label{ms}The transverse measure defined by
$$  d\nu^s \equiv  e^{U\circ \sigma} d \nu^{s,L^*}$$
satisfies
\begin{equation}
 \label{stable}
\frac{d{\phi_{-t }}_* \nu^s_{\phi_t(z)}} {d \nu^s_z}(z)  
 = e^{ \int_{0}^t  F_{\nu^{u,L}} (\phi_{s}(z)) ds}
\end{equation} 
\begin{proof} We use (\ref{evol*}),  Proposition \ref{livsic} and the reversibility of the flow to get
\begin{equation}
\begin{split}
\frac{d{\phi_{-t }}_* \nu^s_{\phi_t(z)}} {d \nu^s_z}(z) & = \frac {e^{U(\sigma(\phi_{t }(z))}}{e^{U(\sigma(z))}} \frac{d{\phi_{-t }}_* \nu^{s,L^*}_{\phi_t(z)}} {d \nu^{s,L^*}_z}(z) \\
& = \frac {e^{U(\sigma(\phi_{t }(z))}}{e^{U(\sigma(z))}} e^{ \int_{-t}^0F_{\nu^{u,L^*}}(\phi_{s}(\sigma(z))) ds}\\
& = \frac {e^{U(\sigma(\phi_{t }(z))}}{e^{U(\sigma(z))}}  e^{ \int_{-t}^0 (F_{\nu^{u,L}}\circ \sigma + X. U)(\phi_{s}(\sigma(z)) ds}\\
& = e^{ \int_{-t}^0 (F_{\nu^{u,L}}\circ \sigma)(\phi_{s}(\sigma(z)) ds}\\
&= e^{ \int_{-t}^0 F_{\nu^{u,L}}(\phi_{-s}(z)) ds}\\
\end{split}
\end{equation}
Then change $s$ to $-s$.
\end{proof}
\end{proposition}

\subsection{ Flow invariance}We can now prove our main theorem \ref{main}.
\begin{proof}
Let us consider T an open disc transverse to the flow and 
a local flow box. The  stable leaves intersected with T are transversal to the strong unstable foliation. We equip the strong unstable foliation with $ \nu^{u,L}_z$ which is in the Lebesque class.
We do the same for the intersections with unstable leaves and use $\nu^s$ to give a measure to a strong stable foliation.

Now we use (\ref{stable}) and (\ref{varmeasureL}) to check that the measure 

$$ d\mu_{SRB} = dt\otimes d\nu^{u,L} \otimes  d\nu^s$$
is flow invariant. By compactness, the measure is finite,  so up to a scalar it is an invariant probability measure. It  also fulfills a characterization of the SRB measure. 
The conditional measures $dt\otimes d\nu^{u,L}$ of that measure along the unstable leaves have been designed to be  in the Lebesgue measure class with H\"older densities. That measure is absolutely continous with respect to both stable and unstable as a  product. The trick of Anosov (\cite{An1}) can be used to show that the measure is ergodic and mixing.
\end{proof}

\section{ Counting of Gibbs periods,  \ref{minimum} and \ref{growth}}
There is a proof given in the Riemannian case by F. Ledrappier which is valid for any equilibrium state and which, as we will see, extends verbatim to our setting.
Using the construction of the transverse measure from a quasi-invariant measure $\mu$, Ledrappier shows in Prop 2 of \cite{Le} the deep fact that the pressure of the measure has to vanish;
 \begin{equation}
 P(- F_{\mu})=0.
 \end{equation}
  As a very  nice corollary, he obtained  the following result that  we will extend to our settings.
 \begin{corollary} (Ledrappier \cite{Le} Corollary 1)
 Let  $c_{\mu}$ be a non-trivial (non cohomologous to a constant)  H\"older cocycle  associated to a quasi-invariant measure $\mu$ on the boundary $\partial \Omega$ and $F_{\mu}$ the $ C^{0+\kappa}$ associated function as in (\ref{ctoF}). If all the periods are positive, the vanishing of the pressure $P(- F_{\mu})=0$ implies

 \begin{align}
  \label{limsup}
 1  &  = \limsup_{ n \rightarrow +\infty} \frac {1}{n} \log ( \# \lbrace \g \in \G  ; n-1 \leq c_{\mu}(\g,\g^+) < n \rbrace)  \\
       & = \lim _{ n \rightarrow +\infty} \frac {1}{n} \log ( \# \lbrace \g \in \G  ;  c_{\mu}(\g,\g^+) \leq n \rbrace
 \end{align}
 and 
 \begin{equation}
  \label{K}
 \inf_{\g \neq 1}\frac{c_{\mu}(\g,\g^+)}{ \ell(\gamma)} > 0
 \end{equation} where $\ell(\g)$ is the Hilbert length of $\rho(\g)$.
 \end{corollary}
\noindent This proves both the second item of  proposition \ref{minimum} and Proposition \ref{growth} in our two examples.

Even if his result is proved in a Riemannian context  we can reprove it here assuming only that the function is H\"older continuous to provide some details  missing in the written version of his paper. 
\begin{proof}
From the general theory for topologically transitive Anosov flows,  the   pressure function for a H\"older cocycle $F$,
 $P(-t F )\equiv P( \phi ) (-tF)$ defined on $ t\geq 0$, satisfies the variational formula
 
 \begin{equation} 
P( \phi)( -tF) = \sup_{\mu \ \phi-{inv \ prob}} \lbrace  h( \phi , \mu) - t \int F d\mu \rbrace
 \end{equation}
It is proven in Parry-Pollicott \cite{Pa-Po}  that  the function $ t \rightarrow P(-tF)$ defined on $ \R^+$ is decreasing  and analytic in $t$.
 In (Prop 4.1.2) they show that this function is strictly convex as long as $F$ 
 is not homologous to a constant. From the definition and the non-negativity  of the entropy we have 
  \begin{equation} 
P( \phi)( -tF) \geq  \sup_{\nu \ \phi-{inv \ prob}} \lbrace  - t \int F d\nu \rbrace =  -t \inf_{\nu \ \phi-{inv \ prob}}\lbrace  \int F d\nu \rbrace
 \end{equation}
Hence if $ K \equiv \inf_{\nu \ \phi-{inv \ prob}}\lbrace  \int F d\nu \rbrace =0  $ the pressure function stays non-negative and cannot vanish at $t=1$ as required. Hence $K>0$.  

Note that $$ \chi_\mu(\g) \equiv  \frac{c_{\mu}(\g,\g^+)}{ \ell(\g)}$$ is the integral of $F_\mu$ with respect to the  Dirac measure of the periodic orbit corresponding to $\g$.
 By the  density of the periodic Dirac measures among the flow invariant probability measures  we conclude for (\ref{K}).
 
 To prove (\ref{limsup}), let us first give some coarse estimates and fix notations.
Let $a_n =  \# A_n = \# \lbrace [\g] \in [\G]  ; n-1 \leq c_{\mu}(\g,x^+(\g)) < n \rbrace$. Set   
\begin{eqnarray}\label{1}
t_0 = \limsup_{ n \rightarrow +\infty} \frac {1}{n} \log (a_n) 
\end{eqnarray}
It is immediate that $t_0 > -\infty $ otherwise the cocycle would be cohomologous to $0$. From (\ref{K}) for any $\g \in A_n $ we know  
$$  \ell( \g)\leq  \frac{c_{\mu}(\g,x^+(\g))}{K} < \frac {n}{K}, $$ 
 which implies that $a_n$ cannot grow faster that the number of periodic orbits of length smaller than $\frac {n}{K}$, hence 
 $$ t_0 \leq \limsup_{ n \rightarrow +\infty}\frac {1}{n} \log \# \lbrace [\g] \in [\G] ; \ell( \g )\leq   \frac {n}{K}\rbrace = \frac{h}{ K} $$
\subsection{ Lower bound} 
Given $c$,  a positive constant sufficiently large we can also introduce  the sets 
$$ A(n,T) \equiv \lbrace [\g] \in [\G] ; T-c \leq \ell(\g) < T ; n-1 \leq c_{\mu}(\g,x^+(\g)) < n \rbrace$$
together with 
$ a(n,T) = \# A(n,T) $.
Obviously $ a(n,T) \leq a_n$.\\
According to (\ref{P}),
\begin{equation*}
 \limsup_{T \rightarrow  +\infty}  \frac {1 } {T} \log   \overset{\infty} {\underset{n=1} \sum}  a(n,T) e^{-t n} \leq P(-tF) \leq \limsup_{T \rightarrow  +\infty}  \frac {1 } {T} \log   \overset{\infty} {\underset{n=1} \sum}  a(n,T) e^{-t( n-1)}
\end{equation*}
where in fact the summation is a finite summation.\\
Thus
\begin{equation}
 P(-tF)=  \limsup_{T \rightarrow  +\infty}  \frac {1 } {T} \log   \overset{\infty} {\underset{n=0} \sum}  a(n,T) (e^{-t})^n
\end{equation}
We may compare it to 
\begin{equation}
P(-tF) \leq  \limsup_{T \rightarrow  +\infty}  \frac {1 } {T} \log   \overset{\infty} {\underset{n=0} \sum}  a_n (e^{-t})^n
\end{equation}
But the radius of convergence of the sum 
$$ \overset{\infty} {\underset{n=0} \sum}  a_n X^n $$ is given by the Hadamard formula
$$ \frac {1}{R}=  \limsup_{ n \rightarrow +\infty} (a_n)^{1/n}. $$ 
From (\ref{1}), 
given $\epsilon$ there exists $N >0 $ such that for all $n > N$, $ \frac {1}{n} \log (a_n) < t_0 +\epsilon $.   Hence
$  a_n < e^{n(t_0 +\epsilon )}$,  which implies that for all positive $\epsilon$ 
$$ \frac {1}{R} \leq e^{t_0 +\epsilon }$$
that is 
$$ R \geq  e^{-t_0}.$$

And we know that if $t<1$ the pressure is positive,  thus the sum $\overset{\infty} {\underset{n=0} \sum}  a_n (e^{-t})^n$ diverges, hence $ e^{-t} \geq R$ that is 
$ t \leq t_0 $,   which concludes
$$ t_0 \geq 1.$$
\subsection{ Upper bound} 

For this we need to introduce some more notations for the set of possible values of $\chi_\mu (\g)$  for a periodic orbit $\g$. By compactness of the underlying manifold we know that the H\"older continuous function $F_\mu$ is bounded, thus there  exists $K' = \sup_{\g }\chi_\mu (\g)$. 

Choose $[ a,b] \subset [ K, K']$, and let's introduce the set
$$ B(a,b) \equiv \lbrace [\g]: a < \chi_\mu(\g)  < b   \rbrace $$ 
Let $$ N ( T, a ,b) = \# B(T,a,b) \equiv \#\lbrace [\g] : a < \chi_\mu(\g)  < b , \ell(\g) \in [ T-c, T]  \rbrace  $$
\begin{lemma}\label{lem} For any $[ a,b] $ as above and $h$ the topological entropy 
\begin{equation}
\limsup_{T \rightarrow  +\infty}  \frac {1 } {T} \log N ( T, a ,b) \leq \inf (b, h)
\end{equation}
\end{lemma}
Again from the vanishing of the pressure,
$\forall \eta >0$ there exists $T_0$ such that $\forall T >T_0$
\begin{equation}
 \frac {1 } {T} \log   \underset{[\g] \in [\G] ,T-c <\ell( \g)< T ,\\  \chi^+ (\g) \in [a, b]}{\sum} e^{- c_{\mu}(\g,\g^+)} \leq \frac {1 } {T} \log   \underset{[\g] \in [\G] ,T-c <\ell( \g)< T }{\sum} e^{- c_{\mu}(\g,\g^+)} \leq \eta
\end {equation}
But  the first summation is  with respect to the set $B(T,a,b)$ for which $ c_{\mu}(\g,\g^+) \leq  \ell(\g)b \leq Tb $.
Hence for any $\eta>0$,
\begin{equation}\label{N}
\frac {1 } {T} \log N(T,a,b) -b \leq  \frac {1 } {T} \log   \underset{[\g] \in [\G] ,T-c <\ell( \g)< T ,  \chi (\g) \in [a, b]}{\sum} e^{- c_{\mu}(\g,\g^+)}\leq \eta
\end{equation}

This precisely means that
$$ \limsup_{T \rightarrow  +\infty}  \frac {1 } {T} \log N ( T, a ,b) - b \leq 0,$$
which proves the lemma. \\
\begin{remark} This estimate does not depend on $a \geq K $ hence 
\begin{equation}
 \limsup_{T \rightarrow  +\infty}  \frac {1 } {T} \log N ( T, K ,b) - b \leq 0 .
\end{equation}
\end{remark}

 Given $\eta$ as above and for any $n > N_0>>1$, let us consider a sequence of adapted uniform  subdivisions of the interval $ [K,K']$ of size $Q_n$ and width  $\ \tau_n = \frac{K'-K}{Q_n}$ that will be chosen thereafter 
and let $S_n$ be such that
$ S_n\tau_n \leq h < (S_n+1)\tau_n$.
These subdivisions will be fitted to use the previous lemma in oder to  give an upper bound of the cardinality of the sets
$$A_n \cap B(K+k\tau_n,K+(k+1)\tau_n)$$
where
\begin{equation}
\g \in A_n \cap B(K+k\tau_n,K+(k+1)\tau_n) \implies        \frac{n-1}{K+(k+1)\tau_n} \leq \ell(\g) \leq \frac{n}{K+k\tau_n}
\end{equation}

We know that the formula for the pressure is true provided $c$ is large enough, thus without restriction we can choose $ cK^2> K'$.
\begin{lemma} 
We can choose $\tau_n >0 $ small enough  so that  for any $k \in [ 0, \frac{K'-K}{\tau_n}]$
\begin{equation}
\frac{n}{K+k\tau_n} -c   \leq \frac{n-1}{K+(k+1)\tau_n} 
\end{equation}
\end{lemma}
\begin{proof}

We need to solve 
\begin{equation*}
\begin{split}
 \frac{c}{n}  &  \geq  \frac{1}{K+k\tau_n} -  \frac{1-\frac{1}{n}}{K+(k+1)\tau_n} \\
 & \geq \frac{ \frac{K}{n}+ \tau_n (1+ \frac{k}{n})}{(K+k\tau_n)(K+ (k+1)\tau_n)} 
 \end{split}
\end{equation*}
If we rewrite it as
\begin{equation*}
 \frac{c}{n} (K+k\tau_n)(K+ (k+1)\tau_n) \geq  \frac{K}{n}+ \tau_n (1+ \frac{k}{n})
\end{equation*}
we observe that both expressions on the left and  on the right are growing with $k$. 
 We have 
$$ \frac{c}{n} (K+k\tau_n)(K+ (k+1)\tau_n) \geq  \frac{cK^2}{n}$$
and for the expression on the right,  because  $ k\leq Q_n$, we have 
$$  \frac{K}{n}+ \tau_n (1+ \frac{\frac{K'-K}{\tau_n}}{n})=   \frac{K}{n}+ \tau_n (1+ \frac{Q_n}{n}) \geq  \frac{K}{n}+ \tau_n (1+ \frac{k}{n})$$
So if we  choose $\tau_n$ such that 
$$ \frac{cK^2}{n} \geq  \frac{K}{n}+ \tau_n (1+ \frac{\frac{K'-K}{\tau_n}}{n}) = \frac{K}{n} +  \tau_n  + \frac{K'-K}{n} =  \tau_n + \frac{K'}{n}  $$
that is
\begin{equation}
 \label{tau}
0<  \tau_n \leq \frac{cK^2- K'}{n}
\end{equation}
then all the inequalities are fulfilled.
\end{proof}
 
Now using  the previous lemma with $\tau_n$ as chosen above,  we have for $T= \frac{n}{K+k\tau_n}$,
$$ A_n \cap B(K+k\tau_n,K+(k+1)\tau_n) \subset  B(\frac{n}{K+k\tau_n}, K+k\tau_n,K+(k+1)\tau_n).$$
Hence we know from the Lemma \ref{lem}  and (\ref{N}) that for any  $n > N_0$ 
\begin{equation}
\# (A_n \cap B(K+k\tau_n,K+(k+1)\tau_n)) \leq e^{ (\inf (K+(k+1)\tau_n, h) +\eta) \frac{n}{K+k\tau_n}}
\end{equation}

 Let $S_n$ be such that
$ S_n\tau_n+K \leq h < (S_n+1)\tau_n+K$.
We can now give an upper estimate for $a_n $
\begin{equation}
\begin{split}
a_n  & =  \# A_n = \sum^{Q_n-1}_{k=0} \# (A_n \cap B(K+k\tau_n, K+(k+1)\tau_n)) \\
 & \leq \sum^{Q_n-1}_{k=0}e^{( \inf(K+(k+1)\tau_n, h) + \eta) \frac{n}{K+ k\tau_n} }\\
 & \leq \sum^{S_n-1}_{k=0}e^{( (K+(k+1)\tau_n + \eta) \frac{n}{K+ k\tau_n} } + \sum^{Q_n-1}_{k=S_n}e^{ (h + \eta) \frac{n}{K+ k\tau_n} }
 \end{split}
\end{equation}
To give a simpler upper estimate  for the  last two summations,  let us remark that   for any $0 \leq k $ and small $\eta>0$
$$ \frac{K+(k+1)\tau_n + \eta } {K+ k\tau_n} \leq \frac{K+\tau_n+ \eta} {K}$$
\noindent Thus also  for $k = S_n$, 
$$\frac{h+ \eta }{ K+ S_n\tau_n} \leq \frac{K +(S_n+1)\tau_n+ \eta }{ K+ S_n\tau_n} \leq  \frac{K+\tau_n + \eta} {K}.$$
And for  $k \geq S_n+1 $ we have 
$$\frac{h+ \eta }{ K+ k\tau_n}< \frac{K+k\tau_n +\eta} {K+k \tau_n} \leq  \frac{K+\eta} {K}\leq \frac{K+\tau_n + \eta} {K}.$$
Hence 
\begin{equation}
\frac{1}{n}\log a_n < \frac{1}{n}\log (Q_n) +  \frac{K+\tau_n+ \eta} {K}  = \frac{1}{n}\log (K'-K) - \frac{1}{n}\log(\tau_n) + \frac{K+\tau_n+ \eta} {K}
\end {equation}
Thus due to the expression for $\tau_n$ in (\ref{tau}),  we have  for any $\eta >0$
$$t_0= \limsup_{n \rightarrow  +\infty} \frac{1}{n}\log a_n \leq  \frac{K+\eta} {K}.$$
Hence $t_0\leq 1$.
%
%
 \end{proof}

 \subsection{Proposition \ref{minimum} for a Lebesgue measure and $\beta$-convexity}\label{beta}
In this section,  we focus on specific  geometric aspects  related to the two Lebesgue measure classes used to define the SRB measure. We observe  that the geometry of divisible convex projective sets sheld some light on  Ledrappier's general result. 

\begin{proof}
 First we recall  the definition of $\beta$-convexity  of a quasi-symmetrically convex set.
\begin{definition}[Def 1.2 in \cite{BIHES}] For a $C ^1$-convex function $F:U\ra \mathbb R$, set $D^F_x(h)=F(x+h)-F(x)- F'(x)h$.

For $\beta\in (2,\infty)$, $F$ is said to be {\bf $\beta$-convex} if $$\inf_{\{(x,h):h\neq 0\}} |h|^{-\beta}D^F_x(h)>0.$$  
\end{definition}
For a given strictly convex   $C^1$-domain $\Omega\subset \mathbb{PR}^m$, one can cover $\partial \Omega$ by small open sets $\cal O_i$ and each $\cal O_i$ is realized by the graph of a strictly convex $C^1$ function $F_i:U_i \ra \mathbb R$ where $U_i$ is a convex open set in $\mathbb R^{m-1}$.
Then set $\beta(\Omega)$ to be $$\max_i\min(\beta: F_i\ \text{is}\ \beta\ \text{convex on all compact sets of}\ \cal O_i)$$
It is proved in Corollary 1.5 (b) of \cite{BIHES} that if $\Omega$ is (Gromov) hyperbolic, then $\partial\Omega$ is $\beta$-convex for some $\beta\in (2,\infty)$.
Then Proposition \ref{minimum} takes its full sense when we relate it to the Lyapunov exponents of periodic orbits. For $\g \in  \G=\rho (\pi_1(M))$, the period  $c_{\mu_{\partial\Omega}}(\g,\g^+) $ is  linked to the sum of  the positive Lyapunov exponents along the orbit freely homotopic to $\g$ by the formula
\label{jacobian}
\begin{equation}
\chi^+(\g) \equiv \sum_{i=1}^{n-1} \chi^+_i (\g)= \frac{c_{\mu_{\partial\Omega} }(\g,\g^+) }{\ell(\g)}=\frac{ n\log(\lambda_1(\g)) +\log( \lambda_{n+1}(\g))}{\ell(\g)} .
\end{equation}
If we set 
$K(\rho) \equiv \inf_{\g} {\sum_{i=1}^{n-1} \chi^+_i (\g)}$,
then  the infimum of the smallest Lyapunov exponents, thanks to (\ref{betai}), satisfies 
 $$  K(\rho) \geq  (n-1)\inf_{\g} \chi^+_1(\g) =(n-1) \frac{2}{\beta_{\Omega}}$$
 So the proof of the second item of Proposition \ref{minimum} is also a direct consequence  from the $\beta$-convexity of divisible convex set,  shown by Benoist and Guichard.
 \end{proof}
 \begin{remark}
In dimension 2 we have from the definition an equality
\begin{equation}
   \frac {2}{ K(\rho)} =  \beta_{\Omega} = \beta _{\G} . 
 \end{equation}
\end{remark}
\subsection{Further disgression around Prop \ref{growth} for the SRB measure of a convex divisible sets }\label{digression}
For the 3rd item in Proposition \ref{minimum},  we just give some complementary dynamical elements, hoping this will bring a new point of view. Our approach also relies  on the vanishing of the pressure which is one of the main feature of the SRB measure reflected in the Ruelle formula
$$h_{SRB}(\phi)= \int \chi^+ d\mu_{SRB},$$ where $\chi^+=\sum \text{dim} E_i\cdot \chi_i^+$, where the Lyapunov exponents here are relative to the SRB  measure. The associated potential is
$$ F= \frac{d}{dt}(\log\det D\varphi_t|_{E^u})|_{t=0}.$$ Thus for any $\g \in \G$
$$ \int_{\g}  F= \log\det D\varphi_{\ell(\g)} |_{E^u} = c_{\mu_{\partial\Omega} }(\g,\g^+) = \chi^+(\g) \times \ell(\g)$$
and thus if we consider the Dirac measure probability $ \mu_\g(\g) =1$ of the periodic orbit we have
$$ \int_{\g}  F d\mu_{\g} = \chi^+(\g).$$
 Furthermore  from Parry and Pollicott (\cite{Pa-Po} Theorem 9.4)
we know that  its SRB entropy is
\begin{eqnarray*}
h_{SRB}=\int \chi^+ d\mu_{SRB}=\lim_{T\ra\infty}\frac{1}{Z_T}\sum_{\ell(\gamma) \leq T} \frac{1}{J_\gamma} \chi^+(\gamma)
\end{eqnarray*}
where the Jacobian $J_\gamma=e^{\chi^+(\gamma) \ell(\gamma)}$ and  where
$$Z_T=\sum_{\ell(\gamma)\leq T} \frac{1}{J_\gamma},$$
which immediately implies that $\inf \chi^+(\gamma) \leq h_{SRB}$. \\

\section{ Geodesic current}\label{current}
\subsection {Ledrappier-Pollicott-Schapira cocycle}
Let $F:T^1\tilde M\ra \R$ be a continuous H\"older $\Gamma$-invariant function where $M=\tilde M/\Gamma$ is a negatively curved compact manifold, or a strictly convex real projective manifold. A cocycle  $B^F:\partial\Gamma \times \tilde M \times \tilde M \ra \R$ introduced by Schapira is defined by
$$B^F_z(p,q)\equiv \lim_{s\ra\infty}[\int_0^{s+B_z(p,q)} F(\phi_t(p,z))dt-\int_0^{s} F(\phi_t(q,z))dt],$$ where $B_z:\tilde M\times \tilde M\ra \R$ is  a Busemann function.
When $F=1$, $B_z^F(p,q)=B_z(p,q)$, the Busemann cocycle. It satisfies
the following properties;
\begin{enumerate}
\item $B^F$ is $\Gamma$-invariant,
\item $B^F_z(p,q)=B^F_z(p,o)+B^F_z(o,q)$ for $o,p,q\in \tilde M$ and $z\in\partial\Gamma$
\item if $q$ is on the geodesic ray from $p$ to $z$,
$$B^F_z(p,q)=\int_p^q F$$ where the integral is along the geodesic connecting $p$ and $q$.
\end{enumerate}
To see (3),  for a point $q$ lying on the ray from $p$ to $z$, using $\varphi_{B_z(p,q)}(p,z)=(q,z)$,

$$\lim_{s\ra\infty}[\int_0^{-B_z(p,q)+s+B_z(p,q)} F(\varphi_t(p,z))-\int_0^{-B_z(p,q)+s}F(\varphi_t(q,z))]$$$$=\lim_{s\ra\infty}[\int_0^{B_z(p,q)} F(\varphi_t(p,z))]+\int_0^{-B_z(p,q)+s} F(\varphi_t(q,z))-\int_0^{-B_z(p,q)+s}F(\varphi_t(q,z))]$$$$=\int_p^q F.$$
Fixing a base point $o\in \tilde M$, one obtains a H\"older cocycle $c_F:\Gamma\times\partial\Gamma\ra\R$
$$c_F(\gamma, z)=B^F_z(\gamma^{-1}o,o).$$
Then its period is given by
$$c_F(\gamma,\gamma^+)=B_{\gamma^+}^F(\gamma^{-1}o,o)=B^F_{\gamma^+}(o,\gamma o).$$
\begin{lemma}$c_F(\gamma,\g^+)=\int_{[\gamma]} F$ where $[\gamma]$ is a geodesic representing $\gamma$.
\end{lemma}
\begin{proof}Let  $L_\gamma$ be a geodesic connecting $\gamma^-$ and $\gamma^+$. Consider the geodesic ray $(o,\gamma^+)$ starting from $o$ in the direction $\gamma^+$. Then its image under the action of $\gamma$ is
$(\gamma o, \gamma^+)$. Note that $\phi_{kB_{\gamma^+}(o,\gamma o)}(o,\gamma^+)=(x_k,\gamma^+)$ where $x_k$ is the point on $(o,\gamma^+)$ with $d(o,x_k)=kB_{\gamma^+}(o,x_1)=kB_{\gamma^+}(o,\gamma o)$. By the invariance of $F$ under the action of $\Gamma$ implies
$$\int_0^{kB_{\gamma^+}(o,\gamma o)} F(\phi_t(o,\gamma^+))=\int_0^{kB_{\gamma^+}(o,\gamma o)} F(\phi_t(\gamma o,\gamma^+)).$$
This shows that 
$$c_F(\gamma,\g^+)=\lim_{k\ra\infty} \int_0^{kB_{\gamma^+}(o,\gamma o)} F(\phi_t( o,\gamma^+))-\int_0^{(k-1)B_{\gamma^+}(o,\gamma o)} F(\phi_t(\gamma o,\gamma^+))$$$$=\lim_{k\ra\infty} \int_0^{B_{\gamma^+}(o,\gamma o)} F(\phi_t(x_k, \gamma^+))=\int_{[\gamma]} F$$ since the geodesic segment $(x_k, x_{k+1})$ gets closer to
$L_\gamma$ exponentially fast.
\end{proof}
From this lemma, $c_F(\gamma,\g^+)=c_F(\alpha \gamma\alpha^{-1},(\alpha \gamma\alpha^{-1})^+)$, that is to say that the period is conjugacy invariant.

The dual cocycle $\bar c_F$ is the one associated to the function $\bar F=F\circ a$ where $a$ is the antipodal map on $T^1M$, whose period is $\bar c_F(\gamma,\g^+)=c_F(\gamma^{-1}, \g^-)$.

One can define a Gromov product on $\partial\Gamma$
$$(x,y)_F=B^{\bar F}_x(o,p)+ B^F_y(o,p)$$ where $p$ is on the geodesic connecting $x$ and $y$. Then it is independent of the choice of $p$, and it satisfies
\begin{eqnarray}\label{produit}
(\gamma x, \gamma y)_F-(x,y)_F=- (\bar c_F(\gamma, x)+c_F(\gamma, y)).
\end{eqnarray}

\subsection{Comparison of cocyles}
First we recall the  $d\tilde B^s_z= -A^H_z$ on each stable leaf (see Section \ref{buse}), and $\frac{d B_\xi(\varphi_t(p), q)}{dt}=-1$ where $B$ stands for a Busemann function. Recall from eqn (\ref{ctoF}), $i_X(d\log G)(p,\xi)=-F(p, \xi)$.
We omit the subscript $\mu$ for simplicity.

We begin with a lemma;
\begin{lemma}
\begin{enumerate}
\item  $\frac{d B^F_\xi(\varphi_t(p), q)}{dt}|_{t=0}=- F(p,\xi)$,
\item  $B^F_\xi(p,q)=\log G(p,\xi)+ D(\xi, q)$.
\end{enumerate}
\end{lemma}
\begin{proof}(1) 
First
$$B^F_\xi(\varphi_\tau(p),q)=\lim_{s\ra\infty}[\int_0^{s+B_\xi(\varphi_\tau(p),q)}F(\varphi_{t+\tau}(p,\xi))dt -\int_0^s F(\varphi_t(q,\xi)) dt]$$
$$=\lim_{s\ra\infty}[\int_\tau^{s+ B_\xi(\varphi_\tau(p),q)+\tau} F(\varphi_u(p,\xi))du  -\int_0^s F(\varphi_t(q,\xi)) dt]\ (u=t+\tau)$$
$$= \lim_{s\ra\infty}[\int_\tau^{s+ B_\xi(p,q)} F(\varphi_u(p,\xi))du  -\int_0^s F(\varphi_t(q,\xi)) dt]  $$
$$=\lim_{s\ra\infty}[\int_0^{s+ B_\xi(p,q)} F(\varphi_u(p,\xi))du  -\int_0^s F(\varphi_t(q,\xi)) dt]-\int_0^\tau F(\varphi_u(p,\xi))du  $$
$$=B^F_\xi(p,q)- \int_0^\tau F(\varphi_u(p,\xi))du.  $$
Then
$$\frac{d B^F_\xi(\varphi_\tau(p), q)}{d\tau}|_{\tau=0}=- F(p,\xi).$$
(2) On the stable leaf corresponding to $\xi$, we defined $i_X(d\log G)(p,\xi)=-F(p, \xi)$.
Hence on each flow line $\log G (\varphi_t(p),\xi)$ and $B^F_\xi(\varphi_t(p),q)$ have the same derivative by (1), hence $B^F_\xi(p, q)=\log G(p,\xi)+ D(\xi, q)$ for some $\Gamma$-invariant H\"older function $D:\partial\Omega \times \Omega\ra\R$ since
 flow lines starting from any $p$ on the stable leaf get closer exponentially fast.
\end{proof}
We set $f(\xi)=D(\xi, o)$ for a fixed base point $o\in \Omega$. It is H\"older since $G(\cdot, \xi)$ and $B^F_\xi$ is H\"older in $\xi$-variable.
\begin{proposition}\label{coc}
\begin{enumerate}
\item $c_F(\gamma, \xi)=c_\mu(\gamma,\xi)- f(\gamma \xi)+f(\xi)$
\item The measure $\mu'=e^{f}\mu$ has a cocyle $c_{\mu'}$ such that
 \begin{itemize}
 \item $c_{\mu'}(\gamma,\xi)=c_F(\gamma,\xi)$
 \item $\gamma^*(e^{f}\mu)=e^{f\gamma-c_\mu-f} (e^{f}\mu)$.
 \end{itemize}
\end{enumerate}
\end{proposition}
\begin{proof}(1) Recall $c_F(\gamma, \xi)=B^F_\xi(\gamma^{-1}o,o)$  and from Eqn (\ref{equivrel})
$$\log G(\gamma p, \gamma \xi)=\log G(p,\xi)-c_\mu(\gamma, \xi).$$
Note by $\Gamma$-invariance of $B^F$,
$$B^F_{\gamma\xi}(\gamma p, q)=B^F_\xi(p,\gamma^{-1}q)=B^F_\xi(p,o)+B^F_\xi(o,\gamma^{-1}q).$$ But the first term is equal to
$$\log G(\gamma p, \gamma\xi)+D(\gamma\xi, q)=\log G(p, \xi)-c_\mu(\gamma,\xi)+D(\gamma\xi, q)$$$$=B^F_\xi(p,q)-D(\xi,q) -c_\mu(\gamma,\xi)+D(\gamma\xi, q).$$
Putting $q=o$, we get
$$B^F_\xi(p,o)-c_F(\gamma,\xi)=B^F_\xi(p,o)+B^F_\xi(o,\gamma^{-1}o)$$$$=B^F_\xi(p,o)-D(\xi,o) -c_\mu(\gamma,\xi)+D(\gamma\xi, o).$$
(2) Since $\mu$ is a quasi-invariant measure with $\gamma^*\mu=e^{-c_\mu(\gamma,\cdot)}\mu$, we calculate
$$\gamma^*(e^{f}\mu)(\xi)=e^{f\gamma(\xi)}\gamma^*\mu(\xi)=e^{f\gamma(\xi)} e^{-c_\mu(\gamma,\xi)}\mu(\xi)$$$$=e^{f\gamma(\xi)-c_\mu(\gamma,\xi)-f(\xi)}e^{f(\xi)}\mu(\xi).$$
Hence
$$c_{\mu'}(\gamma,\xi)=c_\mu(\gamma,\xi)-f \gamma(\xi)+f(\xi).$$
\end{proof}

\subsection{Calculation of Gromov product}
According to Proposition \ref{livsic},  $F\circ\sigma= F^*+ L_X U$ for some H\"older function $U$ on $HM$. Then we calculate
$$B^{F\circ \sigma}_z(p,q)=B^{F^*}_z(p,q)+\lim_{s\ra\infty}[ \int_0^{s+B_z(p,q)}L_XU(\varphi_t(p,z))dt-\int_0^s L_XU(\varphi_t(q,z)) dt]$$
$$=B^{F^*}_z(p,q)+\lim_{s\ra\infty}[U(\varphi_{s+B_z(p,q)}(p,z)-U(p,z)- (U(\varphi_s(q,z))-U(p,z))]    $$
$$=B^{F^*}_z(p,q)+U(q,z)-U(p,z).$$
Hence
\begin{align}\label{coo}
\begin{split}
c_{F^*}(\gamma, x)&=B^{F^*}_x(\gamma^{-1}o,o)=B^{F\circ \sigma}_x(\gamma^{-1}o,o)-U(o,x)+U(\gamma^{-1}o,x)\\
&=c_{F\circ \sigma}(\gamma,x)-U(o,x)+U(o,\gamma x)
\end{split}
\end{align} by $\Gamma$-invariance of $U$. Furthermore by (\ref{coo}) and (\ref{produit})
\begin{align}\label{equu}
\begin{split}
&(\gamma x, \gamma y)_F-(x,y)_F\\
=- c_{F\circ \sigma}(\gamma, x)-c_F(\gamma, y)&=
-c_{F^*}(\gamma, x)-U(o,x)+U(o,\gamma x)- c_F(\gamma, y).
\end{split}
\end{align}

Now define the {\bf Gromov product} as
\begin{eqnarray}\label{gromovproduct}
(x,y)_\mu=(x,y)_F - D_{\mu^*}(x,o) - D_\mu(y,o)-U(o,x)
\end{eqnarray} and verify that
$$\gamma^*(e^{-(x,y)_\mu} \mu^* \otimes \mu)=e^{-(\gamma x, \gamma y)_\mu} \gamma^* \mu^*\otimes \gamma^*\mu$$$$=e^{-(\gamma x, \gamma y)_\mu} e^{-c_{\mu^*}(\gamma, x)}e^{-c_\mu(\gamma, y)} \mu^*\otimes \mu=e^{-(x,y)_\mu} \mu^* \otimes \mu.$$
This follows from
$$(\gamma x, \gamma y)_\mu- (x,y)_\mu$$$$=(\gamma x, \gamma y)_F- D_{\mu^*}(\gamma x, o)- D_\mu(\gamma y, o)-U(o,\gamma x)- (x,y)_F+ D_{\mu^*}(x,o)+ D_\mu(y,o)+U(o,x)$$
$$= -c_{F^*}(\gamma x)- c_F(\gamma, y)- D_{\mu^*}(\gamma x, o)- D_\mu(\gamma y, o)+ D_{\mu^*}(x,o)+ D_\mu(y,o) $$
$$= -c_{\mu^*}(\gamma, x)- c_\mu(\gamma,y),$$ where the second equality follows
from  Eqn (\ref{equu}), and the last equality from Proposition \ref{coc} (1).
We remark that if we change the base point $o$ to $o'$, 
$$(x,y)_{o'}=B^{\bar F}_x(o', p)+B^F_y(o', p)-D_{\mu*}(x,o')-D_\mu(y,o')- U(o',x)$$ hence
$$(x,y)_o - (x,y)_{o'}=B^{\bar F}_x(o,o')+B^F_y(o, o')- D_{\mu*}(x,o)+ D_{\mu*}(x,o')- D_\mu(y,o)+ D_\mu(y,o')$$$$- U(o,x)+ U(o', x).$$
This implies that
$$\frac{d(e^{-(x,y)_{o'}} \mu^*\otimes \mu)}{d(e^{-(x,y)_o} \mu^* \otimes \mu)}(x,y)=e^{(x,y)_o- (x,y)_{o'}}.$$  Hence $(x,y)_o- (x,y)_{o'}$ behaves like a generalized Busemann function compared to the usual Patterson-Sullivan measures in  negatively curved Riemannian manifolds.
\subsection{Patterson-Sullivan measure}\label{PS}
Let us recall classical  facts about the Patterson-Sullivan measure associated to  a H\"older function $F$ defined on $H\tilde M$ . For $s>0$, one can form a Poincar\'e series
$$P_s(x,y)=\sum_{\gamma\in \pi_1(M)} e^{-s\int_x^{\gamma y}} F,$$ where the integral is along the geodesic connecting $x$ and $\gamma y$. 
From Ledrappier  \cite{Le} Lemma 2, we know that if  $F$ is such that  $P(-F)=0$ and $\int_{\gamma} F \geq 0$ for all  $\gamma$ in  $\Gamma$, then the Poincar\'e series  $P_{s}(x, y)$ converges for  large $s$.

From (\ref{K}) we know that  $ \inf_{\g \neq 1}\frac{c_{\mu}(\g,\g^+)}{ \ell(\gamma)}=K > 0$. By the specific property, there exist an $R>0$ such that for any geodesic segment  $[ x, \gamma y ]$ obtained by projection to the base manifold, there exists a periodic geodesic $\bar{\gamma}$ of length less than $d(x, \gamma y)+R$ tracking the geodesic segment $[x, \gamma y]$ and  $\int_{x}^{\gamma y} F-c_\mu(\bar{\gamma},\bar\g^+)$ is bounded. The number of  closed geodesic that can be similarly associated to  $\bar{\gamma}$ is bounded by $C\ell(\bar{\gamma}) \leq C \frac{c_\mu(\bar\gamma,\bar\g^+)}{K}$, for some  $C$. Hence  the exponential growth rates are the same.
The critical exponent $\delta$ is defined as $\inf_s \{s>0: P_s(x,y) <\infty\}$. Suppose $P_\delta(p,p)=\infty$. If not, one can modify the sum a little bit, see \cite{Le} for details, to make it work. Then the Patterson-Sullivan measure is obtained as a weak* limit of the measures
$$\mu_x^s=\frac{\sum_{\gamma\in \pi_1(M)}e^{-s\int_x^{\gamma x} F }\delta_{\gamma x}}{P_s(x,x)},$$ where $\delta_{y}$ is a dirac measure at $y$, as $s\ra \delta$, which converges
weakly on $\Omega \cup \partial \Omega$ to a probability Patterson-Sullivan measure $\mu=(\mu_x)$. Then it satisfies
$$\frac{d\g_*\mu_x}{d\mu_x}(\xi)=e^{-h_cc(\g^{-1},\xi)},$$ for a H\"older cocycle $c$ and
the exponential growth rate $h_c$ of $c$ defined by
$$h_c=\limsup_{s\ra\infty}\frac{\log\#\{[\g]:c(\g,\g^+)\leq s\}}{s}.$$

We can summerize previous sections as follows. For a given quasi-invariant measure $\mu$ on $\partial \tilde M$, we can get a potential $F_\mu$ as in Section \ref{pot}.
Then the Ledrappier-Pollicott-Schapira cocyle $c_{F_\mu}$ is related to the measure cocyle $c_\mu$ as in Proposition \ref{coc}.
\begin{theorem}For a closed manifold $M$, consider a quasi-invariant measure $\mu$ on $\partial \tilde M$ with positive period $c_\mu(\gamma,\g^+)>0$. Then by the result of Ledrappier, the H\"older potential $F=F_\mu$ satisfies $P(-F)=0$.
If there exists a quasi-invariant measure $\bar\mu$ with periods $c_{\bar\mu}(\gamma,\g^+)=c_\mu(\gamma^{-1},(\g^{-1})^+)$, then there exists a Gromov product $(x,y)_o$ (depending on $o\in \tilde M$) such that
$$e^{-(x,y)_o} d\mu(x)\otimes d\bar\mu(y)$$ is $\Gamma$-invariant, hence defines a geodesic current.
\end{theorem}
By Proposition \ref{coc}, note that $c_F(\gamma, \xi)=c_\mu(\gamma,\xi)-D(\gamma\xi, o)+ D(\xi,o)$ and
the measure $\mu$ is equivalent to $\nu=e^{D(, o)}\mu$ with $c_\nu=c^F$.
Then it satisfies
$$\frac{d\gamma^* \nu}{d\nu}(\xi)=e^{-c_\nu(\gamma,\xi)}.$$ Hence the measure $\nu$ is a Patterson-Sullivan measure.

\subsection{GPS system}
A continuous Gromov-Patterson-Sullivan (GPS) system is a triple $(\sigma,\bar\sigma,G)$ where $\sigma,\bar\sigma:\Gamma\times M\ra \R$ are continuous proper cocycles and $G:M^{(2)}\ra \R$ is a continuous function such that
$$\bar\sigma(\gamma,x)+\sigma(\gamma, y)=G(\gamma x, \gamma y)- G(x,y)$$ for all $\gamma\in \Gamma$ and $(x,y)\in M^{(2)}$. See \cite{BCZZ} for details.

In our case, $\mu,\bar\mu$ both satisfy
$$\frac{d\gamma_*\mu}{d\mu}(\xi)=e^{-c_\mu(\gamma^{-1}, \xi)},
\frac{d\gamma_*\bar\mu}{d\bar\mu}(\xi)=e^{-c_{\bar\mu}(\gamma^{-1}, \xi)}$$ and 
$m=e^{-(x,y)_o}d\mu\otimes d\bar\mu$ is a $\Gamma$-invariant geodesic current.
Since $e^{-(\xi,\eta)_o}d\mu(\xi)\otimes d\bar\mu(\eta)$ is  $\Gamma$-invariant,
  it is satisfied
$$c_\mu(\gamma, x) + c_{\bar\mu}(\gamma, y)=(x,y)_o-(\gamma x, \gamma y)_o$$ for any $(x,y)\in \partial\Omega^{(2)}$.
Hence $$-(x,y)_o=-(x,y)_F + D_{\bar\mu}(x,o) + D_\mu(y,o)+U(o,x)$$ plays a role of $G$ in GPS triple. See Eqn (\ref{gromovproduct}).

\subsection{ H\"older reparametrization}
We follow the synchronization trick of Parry (\cite{pa} )
For a given H\"older potential $H>0$ defined on $HM$, one can reparametrize the geodesic flow $\phi_t$ as follows. Let 
$$\kappa(x,t)=\int_0^t H(\phi_s(x)) ds.$$ Since $H>\epsilon>0$, $\kappa(x,\cdot)$ is an increasing homeomorphism of $\mathbb R$, and there is an inverse $\alpha$ such that
$$\alpha(x, \kappa(x,t))=k(x, \alpha(x,t))=t.$$
The H\"older reparametrization $\phi^F$ is given by
$$\phi^H_t(x)=\phi_{\alpha(x,t)}(x).$$
Then it is easy to check the period of the periodic orbit $\gamma$ for the flow $\phi^F$
 is  $$\lambda(F,\gamma)=\int_0^{\lambda(\gamma)} H(\phi_s(x))ds$$ where $\lambda(\gamma)$ is the
 period of $\gamma$ for $\phi$.
 
The function  $F=\frac{d}{dt}(\log\det D\varphi_t|_{E^{u,0}})|_{t=0}$
is  not a positive function in general but by the Anosov property there exist $ \tau >0$ such that the function
\begin{equation}
\tilde F(z)= 1/\tau \int _0^{\tau} F(\phi_s(x)) ds = \log\det D\varphi_\tau{}|_{E^{u,0}})|   >0
\end{equation}
is positive and  have the same periods as $ F$.
This come from the fact that after a uniform and sufficiently large $\tau$ the flot is strictly expanding. 
We can  thus produce 
$\phi^{\tilde F}$ is  new  flow on $HM$
 with the same orbits, but with periods
 $$\tilde T(\g)= \int_0^{\ell(\gamma)}\tilde F(\phi_s(x)) ds=\int_\gamma F=(\log\det D\varphi_{l(\g)}|_{E^{u,0}})|=c_\mu(\gamma,\gamma^+)$$
\begin{remark} Being only a time change, this flow preserves the same stable and unstable leaves. For each $p$ on a periodic orbit $\g$ we may choose a local transversal $\Sigma p$ to the flow in $W^{u,0}_p$. This local section is also transverse to the new flow and the return map is the same the orbits being the same. Hence 
 $$  \log\det D\varphi_{T(p)}|_{T\Sigma_p} = \log\det D\tilde \varphi_{\tilde T(p)}|_{T\Sigma_p} $$
 This implies that for any periodic orbit $\chi_{\tilde \varphi}^+ (\g)$ the sum of positive liapunov exponents is identically $1$ because we have proven that 
  $$   \det D\tilde \varphi_{\tilde T(p)}|_{T\Sigma_p}  = \exp( \tilde T(\g)) $$
  \end{remark}
 This is in fact true along any orbit
 \begin{proposition} The  synchronised  flow with respect to the H\"older function $\tilde F $ is such that for any  $z \in HM $
 \begin{equation} 
  \chi^+_{\tilde \varphi} (z) = \lim_{\tilde t \rightarrow \infty} {\frac{1} {\tilde t}}{\log(det (D \tilde \varphi _{\tilde t})\vert_ {E^{u,0}})} =1
 \end{equation}
 \end{proposition}
\begin{proof} 
Since $\tilde\varphi_{\tilde t}(p)=\varphi_t(p)$ where $\tilde t=\int_0^t \tilde F(\varphi_s(p)) ds$, 
$$\tilde X= \frac{d\tilde \varphi_{\tilde t}(p)}{d\tilde t}=\frac{dt}{d\tilde t}\frac{d\varphi_t(p)}{dt}=\frac{1}{\tilde F}X$$ where $m=\frac{1}{\tilde F}$ is H\"older and $C^1$ along the flow.
In general  one can show that 
$$\frac{d}{d\tilde t}(\det D\tilde \varphi_{\tilde t})=m \frac{d}{dt}(\det D\varphi_t)+\frac{dm(\varphi_t)}{dt}   $$  where $m(\varphi_t(z))=\frac{1}{\tilde F(\varphi_t(z))}$.
Hence
$$ \chi^+_{\tilde \varphi} (z) = \lim_{\tilde t \rightarrow \infty} {\frac{1} {\tilde t}}{\log(det (D \tilde \varphi _{\tilde t})\vert_ {E^{u,0}})}=\lim_{\tilde t \rightarrow \infty} {\frac{1} {\tilde t}}\log \int_0^{\tilde t} \frac{d}{d\tilde s}\det D\tilde\varphi_{\tilde s} d\tilde s$$
$$=\lim_{\tilde t \rightarrow \infty} {\frac{1} {\tilde t}}\log \int_0^{\tilde t} [m\frac{d}{ds}(\det D\varphi_s)+\frac{dm}{ds}] \tilde F ds$$
$$=\frac{t}{\tilde t} \frac{1}{t}\log[\int_0^t \frac{d}{ds}(\det D\varphi_s) ds +m(\varphi_t(z))- m(z)]$$   Since $m(\varphi_t)$ is a bonded function,
$$\lim_{t\ra\infty}\frac{1}{t}\log[\int_0^t \frac{d}{ds}(\det D\varphi_s) ds +m(\varphi_t(z))- m(z)]\ra \chi^+(z)$$ and it suffices to show that 
$\lim_{\tilde t\ra\infty} \frac{t}{\tilde t}\ra \frac{1}{\chi^+(z)}$.

But
$$\frac{\tilde t}{t}=\frac{1}{t}\int_0^t \tilde F(\varphi_s(z)) ds=\frac{1}{t}\int_0^t\frac{1}{\tau}\int_0^\tau F(\varphi_{s+u}(z)) du ds$$
$$=\frac{1}{t\tau}\int_0^\tau\int_0^t F(\varphi_{s+u}(z))dsdu=\frac{1}{t\tau}\int_0^\tau\int_u^{t+u} F(\varphi_{s'}(z))ds'du $$
$$=\frac{1}{t\tau}\int_0^\tau(\int_0^t F ds'-\int_0^u F ds'+\int_t^{t+u} F ds') du.$$
Since $u$ moves from $0$ to $\tau$ and $F$ is bounded, the second and the third integrals are bounded quantities $<B$. Hence we get
$$\lim_{\tilde t\ra\infty}\frac{\tilde t}{t}=\lim_{t\ra\infty}\frac{1}{t\tau}\int_0^\tau(\int_0^t F(\varphi_{s'}(z)) ds' + B )du$$$$=\lim_{t\ra\infty}\frac{1}{t\tau}\int_0^\tau\int_0^t F(\varphi_{s'}(z)) ds'du
=\chi^+   $$
This concludes the proof.
\end{proof}
Now using the Abramov formula it is easy to remark that the new flow has also a SRB measure given by 
\begin{equation} 
d\mu^{\tilde \varphi}_{SRB}=  \frac{1}{\tilde F}\frac {d\mu^{ \varphi}_{SRB}}{\int_{HM}\frac{1}{\tilde F} d\mu^{ \varphi}_{SRB} }
\end{equation}
The  SRB entropy satisfies  the Ruelle equality 
\begin{equation} 
h^{\tilde \varphi}_{SRB} = \int_{HM} \chi^+_{\tilde \varphi} (z) d\mu^{\tilde \varphi}_{SRB} = 1
\end{equation}
Similarly, the topological entropy satifies the Ruelle inquality
\begin{equation} 
h^{\tilde \varphi}_{top} \leq \int_{HM} \chi^+_{\tilde \varphi} (z) d\mu^{\tilde \varphi}_{BM} = 1
\end{equation}
hence 
$$h^{\tilde \varphi}_{top}= h^{\tilde \varphi}_{SRB}  =1$$

\subsection{Counting}\label{count}
What follows relies on the work of A. Sambarino \cite{Sa2} (Lemma 1.28) and uses the prime number theorem of  Parry-Pollicott.
Given a rectangular partion $\mathcal R=\{R_1,\cdots, R_N\}$ of the Poincar\'e cross-sections $\Sigma$ of the flow $\phi_t$ on $X$,  with the first return map $P:\Sigma\ra\Sigma$;
\begin{definition} For a flow $(X, \varphi_t)$, it is said that the triple $(\Sigma_A, \pi, r)$ is a Markov coding for $\varphi_t$ if $\Sigma_A=\{\omega=(\cdots, \omega_1,\omega_0,\omega_1,\cdots):P(\text{int}\ R_{\omega_n})\cap \text{int}\ R_{\omega_{n+1}}\neq\emptyset\}$ is the set of subshift of finite type, irreducible and aperiodic,   furthermore
$\pi:\Sigma_A \ra  X$, which assigns a point in $R_{\omega_0}$ to $\omega$, and the roof function
 $ r:\Sigma_A \ra \mathbb{R}^+$, which is the first return time of $\pi(\omega)$,
 are H\"older continuous, and
$\pi_r: \Sigma_A\times \mathbb{R}\ra X$ defined by $\pi_r(z,t)=\varphi_t(\pi(z))$
satisfies
\begin{enumerate}
\item  $\pi_r$ is surjective and H\"older continuous
\item  If $\hat r:\Sigma_A \times \mathbb{R}$ is defined by $\hat r(z,t)=(\sigma z, t-r(z))$, then $\pi_r$ is $\hat r$-invariant, where $\sigma$ is the standard forward left shift,
\item The translation flow $\sigma^r_t$ from $ \Sigma_A \times \mathbb{R}/ \hat r$ to itself is semi-conjugate to $\varphi_t$, i.e.
$$\pi_r \sigma^r_t=\varphi_t \pi_r $$
\item $\pi_r:\Sigma_A \times \mathbb{R}/\hat r \ra X$ has a uniformly bounded cardinality for each fibre  and the set where it is injetive has a full measure for any ergodic invariant probability measure with respect to $\sigma_t^r$.
\end{enumerate}
\end{definition}

Since the flow $\varphi $ is Anosov, it is associated to a coding. Now we use the work of A. Sambarino \cite{Sa2},\cite{Sa}.
\begin{proposition} 
(\cite{Sa2}, Lemma 1.28) Let $(\Sigma_A, \pi, r)$ be a Markov coding for   a transitive Anosov  flow $\varphi_t$ on $X$. Consider
 the  reparametrized flow $\tilde\varphi_{\tilde t}$,  with a positive H\"older function  $ \tilde F : X \ra \mathbb{R}^+_{\star} $.
 Set $ f : \Sigma_A \ra \mathbb{R}^+_{\star}$ by
 $$ f(z) = \int_0^{r(z)} \tilde F ( \varphi_s (\pi(z)) ds.$$
 Then $(\Sigma_A, \pi, f)$ is a Markov coding for the flow $\tilde\varphi_{\tilde t}$.
\end{proposition} 
We also know from Benoist \cite{Be}  that the Hilbert flow is topologically mixing. This is due to the Zariski density of $\G$, which implies that the lengths of periodic orbits generate a dense subgroup of $\mathbb{R}$. A.  Sambarino  proves that this property is invariant by a H\"older reparametrization. To conclude the proof of Theorem \ref{growth}, we just need to use the Prime orbit Theorem of Parry and Pollicott.

\begin{theorem} (Parry-Pollicott \cite{Pa-Po 2, Pa-Po})
Let $\Sigma_A$ be an irreducible sub-shift of finite type, a function $ f : \Sigma_A \ra \mathbb{R}^+_{\star}$ H\"older continuous.
Assume that the suspension flow $ \sigma^f_t$ is weakly mixing with topological entropy $h$. Then the number of periodic orbits 
$N(t) = \# \lbrace  [\gamma]  ;   \ell(\g)  \leq r \rbrace $ grows as 

$$ ht \exp (-ht) N(t) \ra 1$$
when $t \ra \infty$.
 \end{theorem}
%
%


\vskip .1 in
\noindent     Patrick Foulon\\ Aix-Marseille Universit\'e, CNRS,\\
Institut de Mathématiques de Marseille (UMR 7373) \\
Site Sud, Campus de Luminy \\
13288 MARSEILLE Cedex 9\\
Marseille, France.\\
\texttt{patrick.foulon\char`\@univ-amu.fr}\\
\vskip .005 in
\noindent     Inkang Kim\\
     School of Mathematics\\
     KIAS, Hoegiro 85, Dongdaemen-gu\\
     Seoul, 02455, Korea\\
     \texttt{inkang\char`\@ kias.re.kr}

\end{document}